\documentclass{amsart}
\usepackage{amsaddr}
\usepackage{graphicx}
\usepackage{tikz}
\usepackage{amsrefs}
\usepackage{subfigure}

\newtheorem{theorem}{Theorem}
\newtheorem{lemma}[theorem]{Lemma}
\newtheorem{corollary}[theorem]{Corollary}
\newtheorem{example}{Example}
\newtheorem{proposition}[theorem]{Proposition}
\theoremstyle{definition}
\newtheorem{remark}[theorem]{Remark}
\newtheorem{definition}[theorem]{Definition}
\newtheorem{problem}[theorem]{Problem}

\newcommand{\eref}[1]{(\ref{e.#1})}
\newcommand{\fref}[1]{Figure \ref{f.#1}}
\newcommand{\pref}[1]{Proposition \ref{p.#1}}
\newcommand{\tref}[1]{Theorem \ref{t.#1}}
\newcommand{\cref}[1]{Corollary \ref{c.#1}}
\newcommand{\lref}[1]{Lemma \ref{l.#1}}
\newcommand{\sref}[1]{Section \ref{s.#1}}

\newcommand{\mycite}[1]{#1 \cite{#1}}

\numberwithin{theorem}{section}
\numberwithin{equation}{section}
\numberwithin{figure}{section}

\newcommand{\R}{\mathbb{R}}
\newcommand{\Z}{\mathbb{Z}}
\newcommand{\N}{\mathbb{N}}

\newcommand{\ep}{\varepsilon}
\newcommand{\cof}{\operatorname{cof}}
\newcommand{\conv}{\operatorname{conv}}
\newcommand{\semi}{\operatorname{semi}}
\newcommand{\dist}{\operatorname{dist}}
\newcommand{\diam}{\operatorname{diam}}
\newcommand{\interior}{\operatorname{int}}
\newcommand{\id}{1}
\renewcommand{\P}{\mathbb{P}}
\newcommand{\E}{\mathbb{E}}

\newcommand{\lalpha}{\beta}
\renewcommand{\tilde}[1]{\widetilde{#1}}
\newcommand{\eps}{\varepsilon}

\newcommand{\red}{\normalcolor}

\newcommand{\nc}{\normalcolor}

\begin{document}

\title{The Limit Shape of Convex Hull Peeling}

\author{Jeff Calder}
\address{Department of Mathematics, University of Minnesota}
\email{jcalder@umn.edu}

\author{Charles K Smart}
\address{Department of Mathematics, The University of Chicago}
\email{smart@math.uchicago.edu}

\begin{abstract}
  We prove that the convex peeling of a random point set in dimension $d$ approximates motion by the $1/(d+1)$ power of Gaussian curvature.  We use viscosity solution theory to interpret the limiting partial differential equation.  We use the Martingale method to solve the cell problem associated to convex peeling.  Our proof follows the program of \mycite{Armstrong-Cardaliaguet} for homogenization of geometric motions, but with completely different ingredients.
\end{abstract}

\maketitle

\section{Introduction}

\subsection{Overview}

The ordering of multivariate data is an important and challenging problem in statistics.  One dimensional data can be ordered linearly from least to greatest, and the study of the distributional properties of this ordering is the subject of order statistics.  An important order statistic is the median, or middle, of the dataset.  In statistics, the median is generally preferred over the mean due to its robustness with respect to noise.  In dimensions $d \geq 2$, there is no obvious generalization of the one dimensional order statistics, and no obvious candidate for the median.  As such, many different types of orderings, and corresponding definitions of median, have been proposed for multivariate data.  One of the first surveys on the ordering of multivariate data was given by \mycite{Barnett}.  More recent surveys are given by \mycite{Small} and \mycite{Liu-Parelius-Singh}.

In his seminal paper, \mycite{Barnett} introduced the idea of convex hull ordering.  The idea is to sort a finite set $X \subseteq \R^d$ into convex layers by repeatedly removing the vertices of the convex hull.  The process of sorting a set of points into convex layers is called convex hull peeling, convex hull ordering, and sometimes onion-peeling, as in \mycite{Dalal}.  The index of the convex layer that a sample belongs to is called its convex hull peeling depth.  This peeling procedure will eventually exhaust the entire dataset, and the convex hull median is defined as the centroid of the points on the final convex layer.  Convex hull ordering is used in the field of robust statistics, see \mycite{Donoho-Gasko} and \mycite{Rousseeuw-Struyf}, and is particularly useful in outlier detection, see \mycite{Hodge-Austin}.

Since affine transformations preserve the convexity of sets, the convex layers of a set of points are invariant under affine transformations.  Using this symmetry, \mycite{Suk-Flusser} use convex peeling to recognize sets deformed by projection.  This is important, for example, in computer vision, where a common task is the recognition of objects viewed from different angles.  There are also some applications of convex hull peeling to fingerprint identification, see \mycite{Poulos-Papavlasopoulos-Chrissikopoulos}, and algorithmic drawing, see \mycite{Hodge-Austin}.

In this paper, we show that the convex layers of a random set of points converge in the large sample size limit to the level sets of the solution of a partial differential equation (PDE).  The solutions of our PDE have the property that their level sets evolve with a normal velocity given by the $1/(d+1)$ power of Gaussian curvature multiplied by a spatial weight.  When the weight is constant, our PDE is known as affine invariant curvature motion, see \mycite{Cao}, and affine flow, see \mycite{Andrews}, and in two dimensions as affine curve shortening flow, see \mycite{Angenent-Sapiro-Tannenbaum}, \mycite{Moisan}, and \mycite{Sapiro-Tannenbaum}.  We use the level-set method of \mycite{Evans-Spruck} to make sense of the limiting equation.

The high level outline of our proof is identical to that of \mycite{Armstrong-Cardaliaguet}, who supplied the prototype for quantitative homogenization of random geometric motion.

\subsection{Main Result}

\begin{figure}[h]
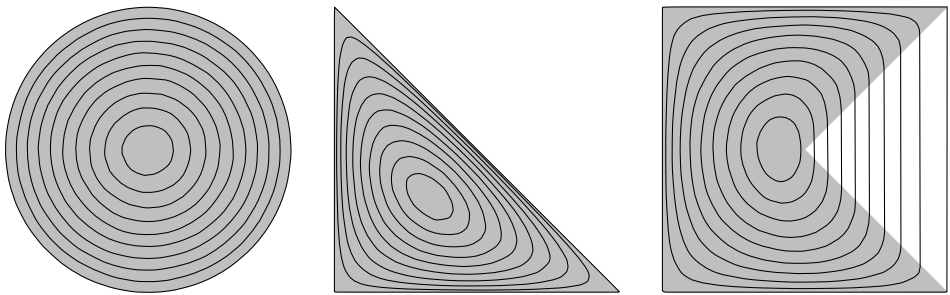

  \input{ball.tikz}
  \hfill
  \input{triangle.tikz}
  \hfill
  \input{nonconvex.tikz}
  \caption{Peels $K_{1+\lfloor k n/10 \rfloor}$ for $k = 0, ..., 9$ of the convex peeling $K_1, ..., K_n$ of $10^5$ points selected independently and uniformly at random from the three different shaded sets.}
  \label{f.peels}
\end{figure}

The convex peeling of a set $X \subseteq \R^d$ is the nested sequence of closed convex sets defined by
\begin{equation*}
  K_1(X) = \conv(X)
  \quad \mbox{and} \quad
  K_{n+1} = \conv(X \cap \interior(K_n(X))),
\end{equation*}
where $\conv(X)$ denotes the convex hull of $X$ and $\interior(K)$ denotes the interior of $K$.  Several examples of convex hull peeling are displayed in \fref{peels}.

It is convenient to encode the convex peeling of $X$ as a function, by stacking the interiors of the peels:
\begin{equation}
  \label{e.hdef}
  h_X = \sum_{n \geq 1} \id_{\interior(K_n(X))}.
\end{equation}
We call $h_X : \R^d \to \N \cup \{ + \infty \}$ is the convex height function of $X$.  We are interested in the shape of $h_X$ for random finite sets $X \subseteq \R^d$.  The starting point of our work is the following result.

\begin{theorem}[\mycite{Dalal}]
\label{t.dalal}
  There is a constant $C > 0$ such that, if $X_n \subseteq \R^d$ consists of $n$ points chosen independently and uniformly at random from the unit ball $B_1$, then $C^{-1} n^{2/(d+1)} \leq \E[\max h_{X_n}] \leq C n^{2/(d+1)}$.
\end{theorem}

We strengthen the above result to
\begin{equation*}
  \E\left[\max h_{X_n} \right] \sim n^{2/(d+1)},
%  \lim_{n \to \infty} \frac{\E[\max h_{X_n}]}{n^{2/(d+1)}} \quad \mbox{exists}.
\end{equation*}
and we show that the rescaled height functions $n^{-2/(d+1)} h_{X_n}$ converge almost surely to the limit
\begin{equation*}
  \alpha \tfrac{(d+1)}{2d} (1 - |x|^{2d/(d+1)}),
\end{equation*}
where $\alpha > 0$ depends only on dimension.  In fact, we prove something stronger:

\begin{theorem}
  \label{t.main}
  Let $U \subseteq \R^d$ be convex, open, and bounded, let $f \in C(\overline U)$ satisfy $f > 0$, and, for $m > 1$, let $X_m \sim Poisson(m f)$.  For every $\ep > 0$, there is a $\delta > 0$ such that
  \begin{equation}
    \label{e.convergence}
    \P[\sup_{\bar U} |m^{-2/(d+1)} h_{X_m} - \alpha h| > \ep] \leq \exp( - \delta (\log m)^{-2} m^{1/3(d+1)} ),
  \end{equation}
  where \red $\alpha>0$ depends only on dimension, \nc and $h \in C(\bar U)$ is the unique viscosity solution of
  \begin{equation}
    \label{e.bvp}
    \begin{cases}
      \langle Dh, \cof(-D^2 h) Dh \rangle = f^2 & \mbox{in } U \\
      h = 0 & \mbox{on } \partial U.
    \end{cases}
  \end{equation}
\end{theorem}

Note that $\cof(A)$ denotes the cofactor matrix of $A$, which is the unique continuous map such that $\cof(A) = \det(A) A^{-1}$ when $A$ is invertible.  The inner product on the left-hand side of the PDE \eref{bvp} is the Gaussian curvature of the level sets of $h$ multiplied by the square norm of the gradient $|Dh|^2$ (see \sref{geom}).  The notion of viscosity solution is defined in \sref{viscosity}.  The notation $X \sim Poisson(f)$ indicates that $X$ is a random subset of $\R^d$ whose law is Poisson with density $f$.

\red We remark that \eref{convergence} gives a quantitative probabilistic estimate, but does not give a convergence rate for $m^{-2/(d+1)}h_{X_m}\to \alpha h$, since we are not able to quantify the dependence of $\delta$ on $\eps$. Indeed, the proof of \tref{main} approximates $h$ by simpler piecewise sub-~and supersolutions, and the number of pieces required depends in some way on the regularity of $h$, which may not be smooth. We refer to Problem \ref{p.fluc} for more dicussion. \nc

While we stated our main Theorem for a Poisson cloud, we can recover from \tref{main} the same result for a sequence of independent and identically distributed (\emph{i.i.d.})~random variables. 
\begin{corollary}
  \label{c.main}
Assume that $\int_U f\,dx = 1$. Let $Y_1,Y_2,Y_3,\cdots$ be a sequence of \emph{i.i.d.}~random variables with probability density $f$ and set 
 \[Z_m = \left\{ Y_1,Y_2,\dots,Y_m \right\}.\] 
For every $\ep > 0$, there is a $\delta > 0$ such that
  \begin{equation}
    \label{e.convergenceiid}
    \P[\sup_{\bar U} |m^{-2/(d+1)} h_{Z_m} - \alpha h| > \ep] \leq \exp( - \delta (\log m)^{-2} m^{1/3(d+1)} ),
  \end{equation}
  where $h \in C(\bar U)$ is the unique viscosity solution of \eref{bvp}. In particular, 
\begin{equation}
m^{-2/(d+1)} h_{Z_m} \longrightarrow \alpha h \ \ \text{uniformly and almost surely as } m\to \infty.
\label{e.as}
\end{equation}
\end{corollary}

There are three natural problems worth mentioning.

\begin{problem}
  Determine the constant $\alpha$ in \tref{main}. When $d=2$ numerical simulations suggest $\alpha=4/3$.
\end{problem}

\begin{problem}\label{p.fluc}
  Determine the scaling limit of the fluctuations $m^{-2/(d+1)} h_{X_m} - \alpha h$ in \tref{main}.  Our proof of \tref{main} does not quantify the dependence of $\delta$ on $\ep$.  The regularity of the limiting $h$ should have an effect on this dependence.  Some results on the regularity of $h$ can be found in \mycite{Andrews} and \mycite{Brendle-Choi-Daskalopoulos}.  However, even in the case $h$ is smooth, we expect our bound is sub-optimal.
\end{problem}

\begin{problem}
  Prove uniqueness of solutions in the case $f \geq 0$.  Our proof of \tref{main} requires $f > 0$.  As we see from the non-convex example in \fref{peels}, the geometric interpretation as motion by a power of Gauss curvature becomes degenerate in the case $f = 0$.  Looking forward to \sref{viscosity}, the uniqueness of solutions depends upon our being able to perturb subsolutions to strict subsolutions.  When $f > 0$, this is easily achieved by homogeneity.  When $f$ is allowed to vanish, strictness must be obtained in a different way.  For example, one could add $\ep |x|^2$ to make the super level sets concave.  However, making such perturbations work in general appears to require curvature bounds for the level sets, which are currently unavailable in our setting.
\end{problem}

\subsection{Game Interpretation}

To formally derive the PDE \eref{bvp}, we observe that, for arbitrary $X \subseteq \R^d$, the height function $h_X$ satisfies the dynamic programming principle:
\begin{equation}
  \label{e.dpp}
\red  h_X(x) = \inf_{p \in \R^d \setminus \{ 0 \}} \sup_{p \cdot (y - x) > 0} [\id_X(y) + h_X(y)] \quad \mbox{for all } x \in \R^d.\nc
\end{equation}
As in \mycite{Kohn-Serfaty}, this leads to an interpretation of the convex height function as the value function of a two-player zero-sum game.

In the convex hull game, the players take turns defining a sequence of points $x_0, p_0, x_1, p_1, x_2, p_2, ... \in \R^d$.  The game starts at a point $x_0 \in \R^d$.  After $x_k$ is defined, player I chooses any $p_k \in \R^d$ satisfying $p_k \neq 0$.  After $p_k$ is defined, player II chooses any point $x_{k+1} \in \R^d$ satisfying $p_k \cdot (x_{k+1} - x_k) > 0$.  Players I and II seek to minimize and maximize, respectively, the final score $\sum_{k \geq 1} \id_X(x_k)$.  In particular, we see that player I seeks to, in the fewest possible moves, isolate play to a half-space that is disjoint from the set $X$.  Meanwhile, player II seeks to land on the set $X$ as often as possible.  \red An optimal choice for player $I$ is to choose $p_k$ so that the halfspace $\{x\, : \, p_k\cdot (x-x_k)>0\}$ is disjoint from $A:=\{x\, : \, h_{X}(x)\geq h_X(x_k)\}$. Since $A$ is convex, such as choice $p_k$ exists, and for this choice we have $h_{X}(x_{k+1})\leq h_X(x_k)-1$ for any feasible choice of $x_{k+1}\in X$ by player II. An optimal choice for player II is to choose $x_{k+1}\in X$ so that $p_k\cdot (x_{k-1}-x_k)>0$ and $h_{X}(x_{k+1})= h_X(x_k)-1$. Such a point $x_{k+1}\in X$ is guaranteed to exist by the definition of convex peeling. Thus, each step of the game moves exactly to the previous convex layer, and the convex height function $h_X(x_0)$ is precisely the final score under optimal play started at $x_0$.

To explain the limiting equation, let $X_m \sim Poisson(m f)$ for large $m > 0$.  Let us assume, even though it is discontinuous, that the rescaled height function $h = m^{-2/(d+1)} h_{X_m}$ is smooth, has uniformly convex level sets, and non-vanishing gradient.  As discussed above, the optimal choice for player I when $x_k=x$ is $p = -Dh(x)$, or any scalar multiple thereof.  Thus, the dynamic programming principle \eref{dpp} becomes
\[\sup_{Dh(x)\cdot (y - x) < 0} [m^{-2/(d+1)}\id_X(y) + h(y) - h(x)] =0. \]
Formally speaking, the dynamic programming principle implies that, on average, the set
\begin{equation*}
  \{ y \in \R^d : Dh(x) \cdot (y - x) < 0 \mbox{ and } h(y) \geq h(x) - m^{-2/(d+1)} \}
\end{equation*}
should have one point.  That is, its probability volume $\int_A f(y)\,dy$ should be proportional to $m^{-1}$.  Taylor expanding $h$ to compute the volume, we obtain
\begin{equation*}
  \langle Dh(x), \cof(-D h(x)) D h(x) \rangle \approx C f(x)^2,
\end{equation*}
for a constant of proportionality $C$. That is, $h$ should satisfy \eref{bvp}, up to the constant $C$, which is not determined by this heuristic argument.
\nc

\subsection{Geometric interpretation}
\label{s.geom}

We can give a precise geometric interpretation of \eref{bvp}. The Gaussian curvature of the level surfaces of $h$ is given  by \mycite{Giga}
\[\kappa_G = \frac{\langle Dh, \cof(-D^2 h)Dh\rangle}{|Dh|^{d+1}},\]
provided $h \in C^2$ and $Dh \neq 0$. Therefore we can formally rewrite \eref{bvp} as
\begin{equation}\label{e.altpde}
|Dh| \kappa_G^\frac{1}{d+1} = f^\frac{2}{d+1}.
\end{equation}
This equation has the property that the level sets $\{h=t\}$ move with a normal velocity given by
\begin{equation}\label{e.vel}
\nu = \kappa_G^\frac{1}{d+1}f^{-\frac{2}{d+1}}.
\end{equation}
To see why, consider nearby level sets $h=t$ and $h=t + \Delta t$. Let $\Delta x$ denote the normal distance between these level sets at some point $x \in \R^d$. Then $|D h(x)| \approx \Delta t/\Delta x$ and hence
\[\Delta x \approx \kappa_G^\frac{1}{d+1}f^{-\frac{2}{d+1}}\Delta t.\]
This implies that $h(x)$ is the arrival time of the boundary $\partial U$ as it evolves with a normal velocity given by \eref{vel}. When $f$ is constant, this geometric motion is known as affine invariant curvature motion, or the affine flow. \mycite{Cao} derived affine invariant curvature motion as the continuum limit of affine erosions. A similar geometric flow (motion by Gauss curvature) was derived by \mycite{Ishii-Mikami} for the wearing process of a non-convex stone.  

\subsection{Representation formulas for solutions}
\label{sec:exact}

 Assume that  $f$ is a radial function, that is  $f(x)=f(r)$ where  $r=|x|$. We look for a solution of \eqref{e.bvp} in the form  $h(x)=v(r)$ where  $v$ is a decreasing function. Using the alternative form \eqref{e.altpde} we see that 
\[v'(r)=-r^{\frac{d -1}{d +1}}f(r)^{\frac{2}{d +1}}.\]   
Integrating and using the boundary condition  $\lim_{r\to \infty}v(r)=0$ we have  
\[v(r)= -\int_r^\infty v'(s)\, ds=\int_r^\infty s^{\frac{d -1}{d +1}}f(s)^{\frac{2}{d +1}}\, ds.\]
Therefore we find that 
\begin{equation}\label{eq:exact}
h(x)=\int_{|x|}^\infty r^{\frac{d -1}{d +1}}f(r)^{\frac{2}{d +1}}\, dr.
\end{equation}
We give some applications of this formula below.
\begin{example}[Uniform distribution on a ball]
Suppose that  $f(x)=\frac{1}{|B_1|}$ for  $x\in B_1$ and  $f(x)=0$ otherwise,  where $B_1$ denotes the unit ball. Then we have 
\begin{equation}\label{eq:ball}
h(x)=\frac{d +1}{2d|B_1|^{\frac{2}{d +1}}}\left( 1 -|x|^{\frac{2d}{d +1}} \right).
\end{equation}
The (normalized) maximum convex depth in this case is 
\[\alpha h(0)=\frac{\alpha(d +1)}{2d|B_1|^{\frac{2}{d +1}}}.\]
\end{example}
\begin{example}[Standard normal distribution]
Suppose that  $f(x)=(2\pi )^{-d/2}e^{ -|x|^2/2}$. Then
\begin{equation}\label{eq:normal}
h(x)=\frac{1}{(2\pi )^{\frac{d}{d +1}}}\int_{|x|}^\infty r^{\frac{d -1}{d +1}}e^{ -\frac{r^2}{d +1}}\, dr.
\end{equation}
The maximum convex depth in this case is 
\[\alpha h(0)=\frac{\alpha }{2}\left( \frac{d +1}{2\pi } \right)^{\frac{d}{d +1}}\Gamma \left( \frac{d}{d +1} \right).\]
\end{example}

Due to the affine invariance of \eqref{e.bvp}, we can scale the solution formula \eqref{eq:exact} by any affine transformation. For example, suppose that 
\[f(x)=|A|f(|Ax +b|),\]
where  $A\in \R^{d \times d}$ is a non-singular matrix, $b\in \R^d$, and $|A|$ is the absolute value of the determinant of $A$.  Then we have 
\begin{equation}\label{eq:exact2}
h(x)=\int_{|Ax +b|}^\infty r^{\frac{d -1}{d +1}}f(r)^{\frac{2}{d +1}}\, dr.
\end{equation}
\begin{example}[Normal distribution]
Suppose that  
\[f(x)=|2\pi \Sigma  |^{ -\frac{1}{2}}\exp\left(  -\frac{1}{2}(x -\mu )\cdot \Sigma  ^{ -1}(x -\mu ) \right),\]
where  $\mu \in \R^d$ is the mean and $\Sigma  \in \R^{d \times d}$ is the covariance matrix.   
 Then
\begin{equation}\label{eq:normal2}
h(x)=\frac{1}{(2\pi )^{\frac{d}{d +1}}}\int_{|\Sigma  ^{ -\frac{1}{2}}(x -\mu )|}^\infty r^{\frac{d -1}{d +1}}e^{ -\frac{r^2}{d +1}}\, dr.
\end{equation}
\end{example}

\subsection{Distribution of points among layers}
\label{sec:pts}

We show here how Theorem \ref{t.main} can be used to deduce the distribution of points among the convex layers. Let 
\begin{equation}\label{eq:N}
N_m(i)=\# \left\{ X_{m}\cap K_i(X_{m})\setminus K_{i +1}(X_{m})  \right\},
\end{equation}
be the number of points on the  $i^{\rm th}$ convex layer for $X_m \sim Poisson(m f)$.\footnote{The discussion is equally valid for a sequence of $m$ \emph{i.i.d}~random variables with probability density $f$ as in \cref{main}.} Note that the  $i^{\rm th}$ convex layer is approximately the level set   $\{h_{X_{m}}=i\} $.   Since  $m^{ -\frac{2}{d +1}}h_{X_{m}}\to \alpha h$ as  $m\to \infty$ it is possible to show that for any  $0<a<b$  
\begin{equation}\label{eq:avg}
\lim_{m\to \infty}\frac{1}{m}\sum_{i=\lfloor am^{\frac{2}{d +1}}\rfloor}^{\lfloor bm^{\frac{2}{d +1}}\rfloor}N_m(i)=\int_{a\leq \alpha h\leq b}f\, dx\quad \text{ almost surely.}
\end{equation}
By the co-area formula and \eqref{e.altpde} we have 
\[\int_{a\leq \alpha h\leq b}f\, dx=\frac{1}{\alpha }\int_a^b\int_{ \left\{ \alpha h=r \right\}}\frac{f}{|D  h|}\, dS\, dr=\frac{1}{\alpha }\int_a^b\int_{ \left\{ \alpha h=r \right\}}f^{\frac{d -1}{d +1}}\kappa _G^{\frac{1}{d +1}}\, dS\, dr,\]
where  $\kappa _G$ denotes the Gaussian curvature of the level set  $\left\{ \alpha h=r \right\}$.   
It is tempting to set  $b -a=m^{ -\frac{2}{d +1}}$ to get  
\begin{equation}\label{eq:pts}
\lim_{m\to \infty}m^{ -\frac{d -1}{d +1}}N_m(\lfloor tm^{\frac{2}{d +1}}\rfloor)=\frac{1}{\alpha }\int_{\{\alpha h=t\}}f^{\frac{d -1}{d +1}}\kappa _G^{\frac{1}{d +1}}\, dS\quad \text{ almost surely}.
\end{equation}
This does not follow directly from Theorem \ref{t.main} and would require a far more careful analysis of the continuum limit. We leave such an analysis to future work, and proceed with discussing applications. For convenience, let us set 
\begin{equation}\label{eq:Ninf}
N(t)=\frac{1}{\alpha }\int_{\{\alpha h=t\}}f^{\frac{d -1}{d +1}}\kappa _G^{\frac{1}{d +1}}\, dS.
\end{equation}

Note that if  $f(x)=f(r)$ is radial, then  $h(x)=h(r)$ and  $\kappa _G=r^{ -(d -1)}$ for $r=h^{-1}(\alpha^{ -1}t)$. Therefore
\begin{equation}\label{eq:radial}
N(t)=\frac{d|B_1|}{\alpha }f(r)^{\frac{d -1}{d +1}}r^{\frac{d(d -1)}{d +1}},\quad \text{  where }r=h^{ -1}(\alpha ^{ -1}t).
\end{equation}
\begin{figure}
\centering
\subfigure[Uniform distribution on a ball]{\includegraphics[trim=30 20 28 20, clip= true, width=0.45\textwidth]{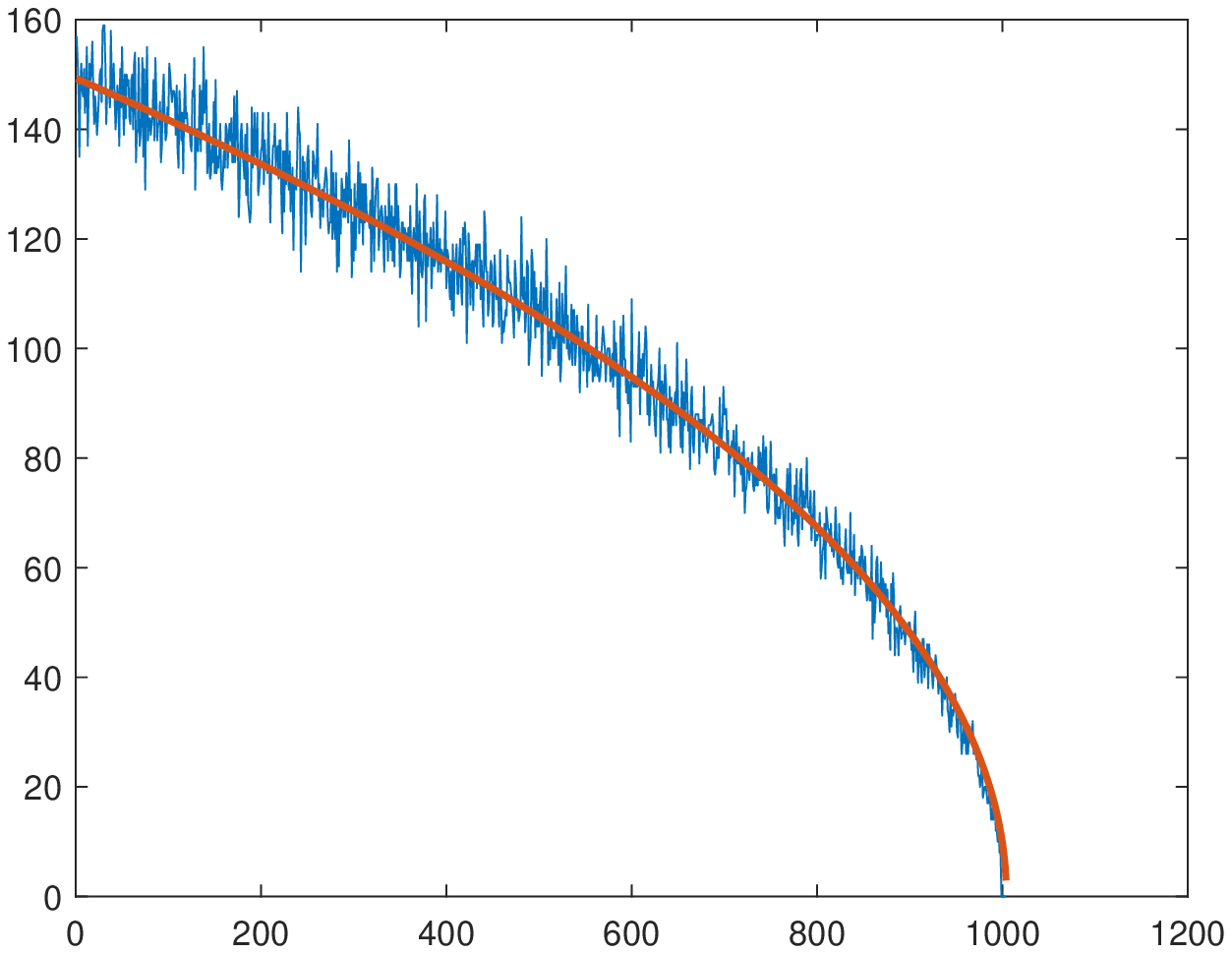}\label{fig:circ}}
\subfigure[Standard normal distribution]{\includegraphics[trim=30 20 28 20, clip= true, width=0.45\textwidth]{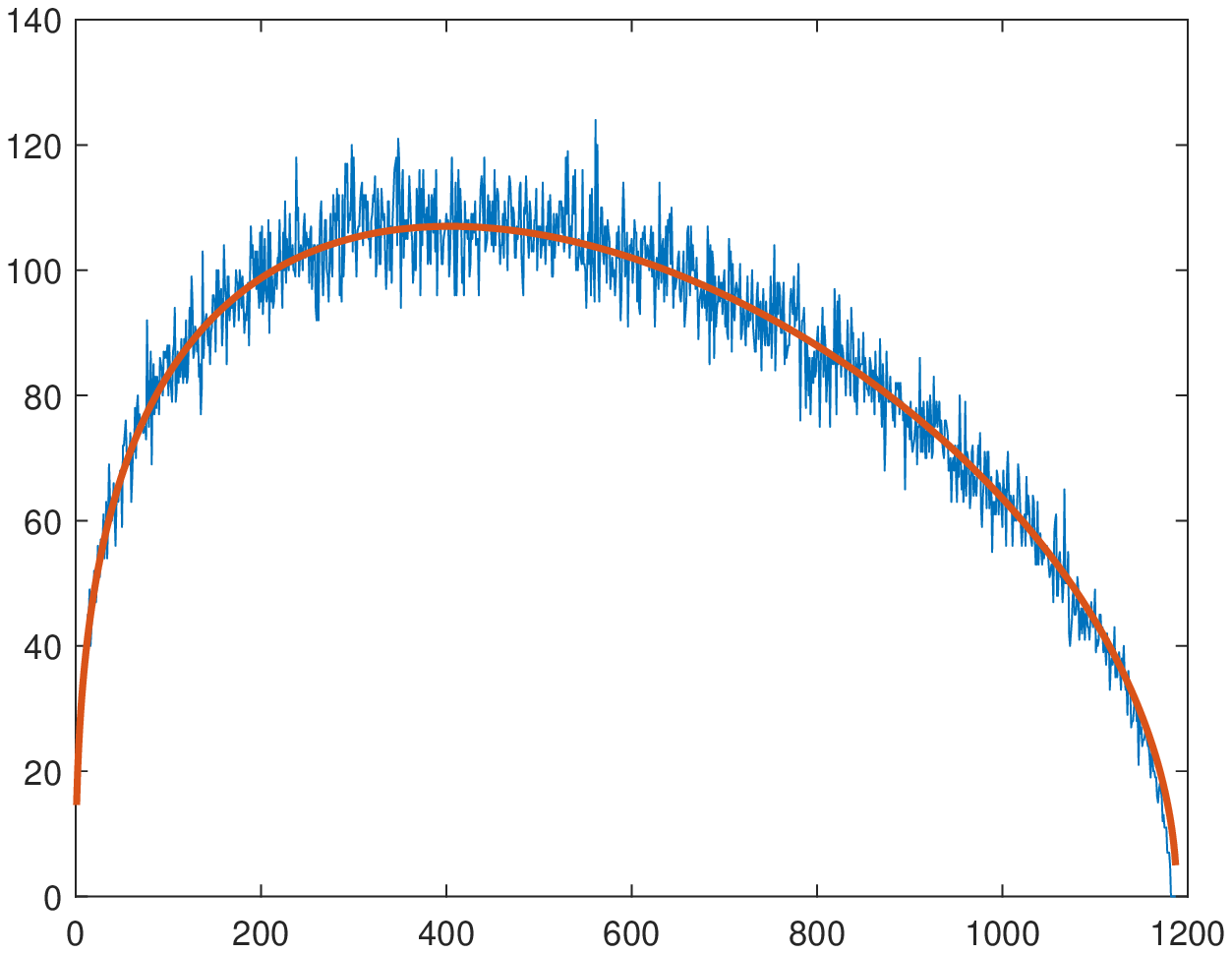}\label{fig:gauss}}
\caption{Comparison of the distribution of points among convex layers with the continuum limit \eqref{eq:pts}. In each figure the vertical axis is the number of points and the horizontal axis is the convex layer index.}
\end{figure}

\begin{example}[Uniform distribution revisited]
For a uniform distribution on the unit ball we have 
\begin{equation}\label{eq:Nball}
N(t)=\frac{d|B_1|^{\frac{2}{d+1}}}{\alpha }\left( 1 -ct \right)^{\frac{d -1}{2}},\quad \text{  where }c=\frac{2d|B_1|^{\frac{2}{d +1}}}{\alpha (d +1)}.
\end{equation}
Figure \ref{fig:circ} shows a simulation comparing $N(t)$ to the distribution of points among convex layers for $n=10^5$ \emph{i.i.d.}~random variables uniformly distributed on the unit ball. Another simulation averaged over 100 trials and shown in Figure \ref{fig:bl} suggests there is a boundary layer phenomenon. The first convex layer has significantly more points than nearby subsequent layers. 
\end{example}

\begin{example}[Normal distribution revisited]
For the standard normal distribution we have 
\begin{equation}\label{eq:Nnormal}
N(t)=\frac{d|B_1|}{\alpha }\left( \frac{r}{\sqrt{2\pi }} \right)^{\frac{d(d -1)}{d +1}}\exp\left( -\frac{r^2(d -1)}{2(d +1)} \right),
\end{equation}
where $r=r(t)$ satisfies $\alpha h(r) = t$, or 
\[ t=\frac{\alpha}{(2\pi )^{\frac{d}{d +1}}}\int_{r}^\infty s^{\frac{d -1}{d +1}}e^{ -\frac{s^2}{d +1}}\, ds.\]
Figure \ref{fig:gauss} shows a simulation comparing $N(t)$ to the distribution of points among convex layers for $n=10^5$ \emph{i.i.d.}~normally distributed random variables.
\end{example}

\begin{figure}
\centering
\subfigure{\includegraphics[trim=30 20 28 20, clip= true, width=0.45\textwidth]{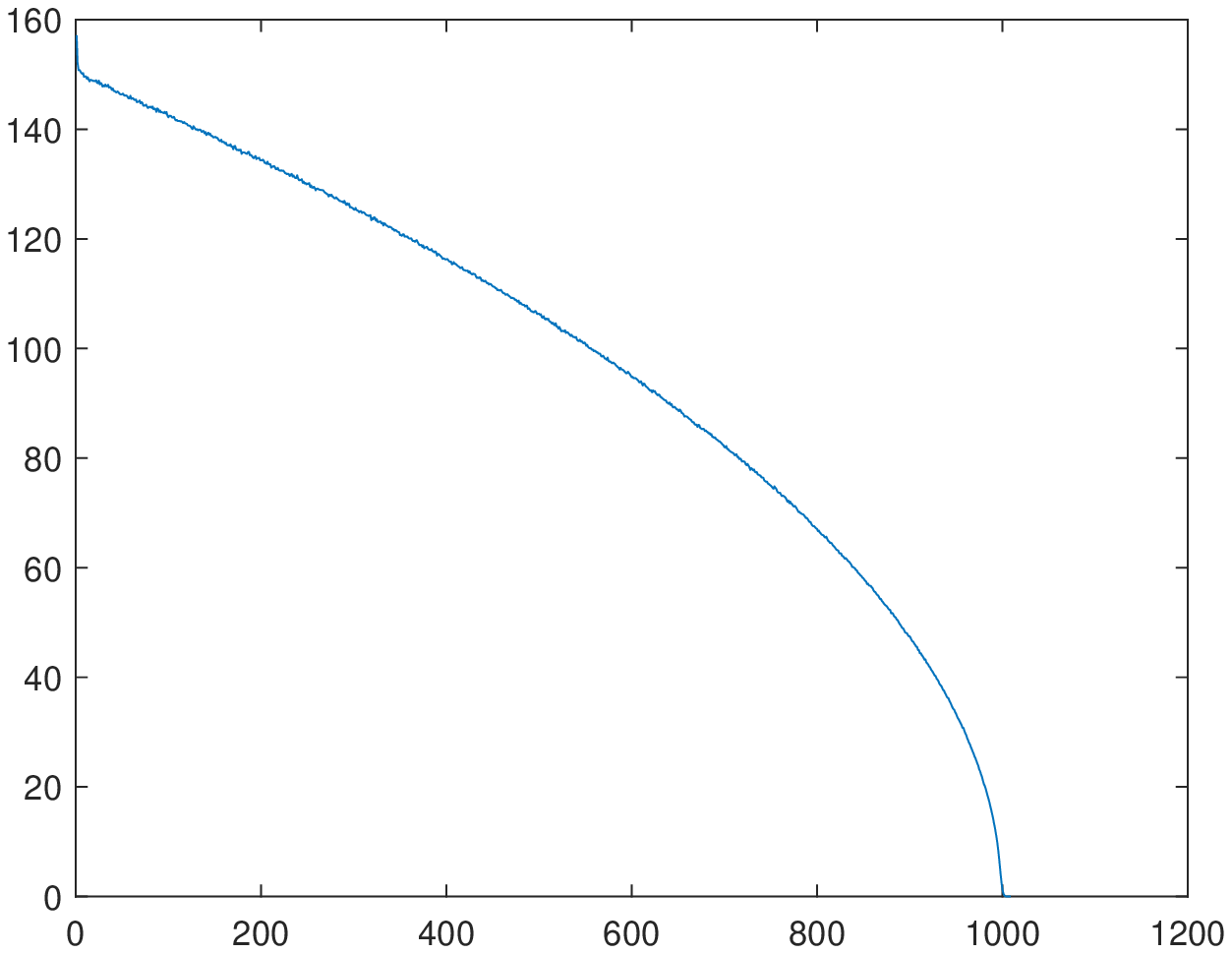}}
\subfigure{\includegraphics[trim=30 20 28 20, clip= true, width=0.45\textwidth]{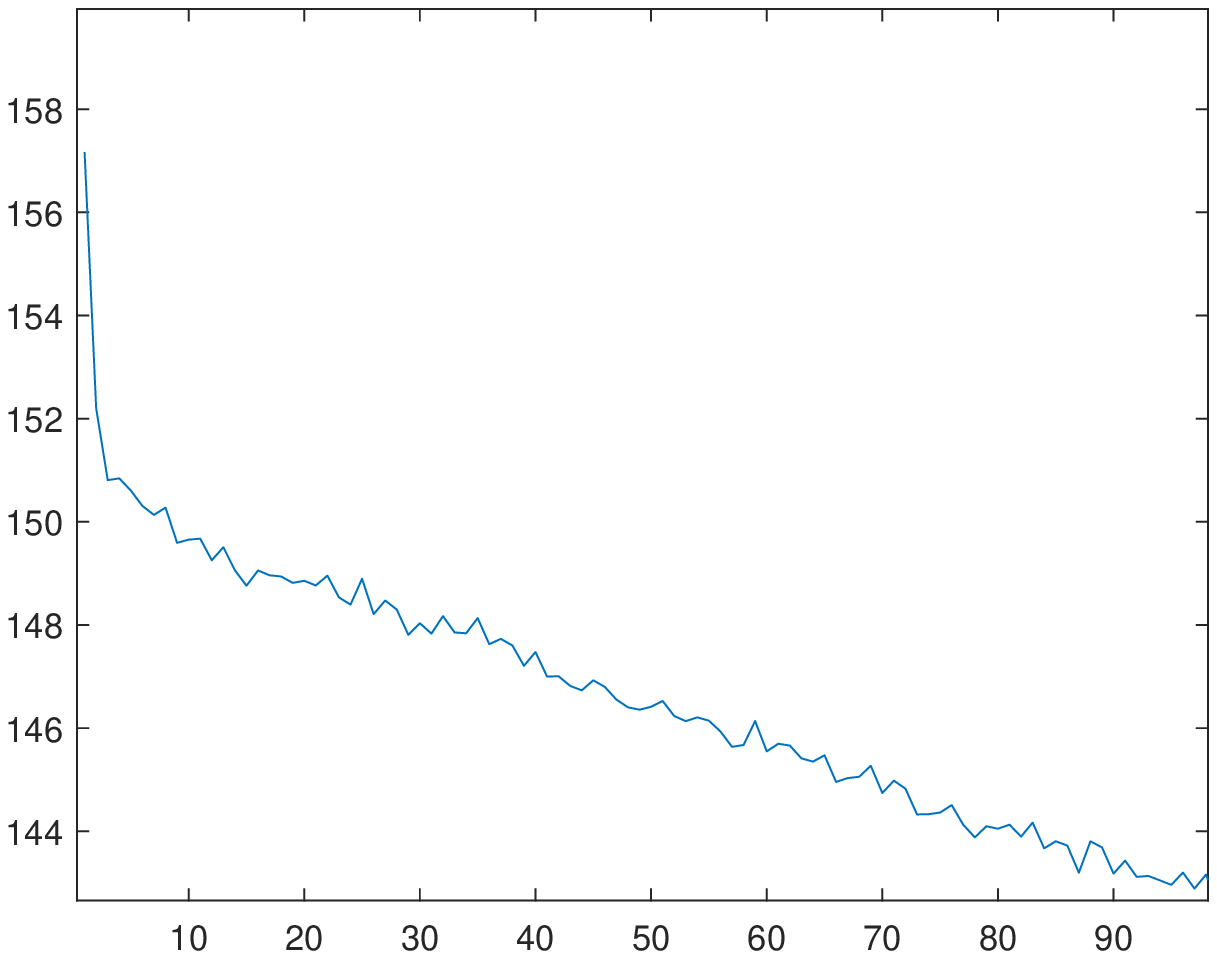}}
\caption{Simulation showing evidence of a boundary layer near the first convex layer. The figure on the right is zoomed in to show the sharp decrease in the number of points between the first and subsequent convex layers.}
\label{fig:bl}
\end{figure}

\red
\subsection{Overview of proof and the cell-problem}
\label{s.overview}

The proof of \tref{main} involves solving a \emph{cell problem}, which in homogenization theory refers to a family of simpler problems whose solutions describe the local behavior of the random function of interest. 
%The solutions of cell problems, which are simpler to solve, are used to prove homogenization, and quantitative rates for cell-problems lead to quantitative rates for homogenization. 
When looking for a cell problem for convex hull peeling, we seek a simpler convex peeling problem that has convenient symmetry and invariance properties, and can locally approximate a general convex peeling problem. 

Let $Y_m \sim Poisson(m)$ and define the standard parabola 
\begin{equation*}\label{e.parabola}
P = \{x\in \R^d \, : \, x_d > \tfrac12 |x^d|^2\},
\end{equation*} 
where
\begin{equation*}\label{e.xd}
x^d = (x_1,\dots,x_{d-1}).
\end{equation*}
The cell problem for convex peeling is the convex peeling of the set $Y_1\cap P$. In particular, we show that almost surely
\begin{equation}\label{e.cellproblem}
\lim_{r\to \infty} \frac{1}{r}h_{Y_1\cap P}(re_d) = \alpha,
\end{equation}
and we establish a convergence rate of $O(\sqrt{r})$ up to logarithmic factors. To see why the function $h_{Y_1\cap P}(re_d)$ should have linear growth in $r$, we note that the set $P \cap \{x_d<r\}$ contains on average $n=O(r^{(d+1)/2})$ points from $Y_1$, and so by \tref{dalal} we expect $O(n^{2/(d+1)})=O(r)$ convex layers in this region.

We call $h_{Y_1\cap P}$ the cell problem for convex peeling due to a family of symmetries that are inherited from convex peeling, which allow us to essentially use \eref{cellproblem} to prove our main result. Indeed, we first note that convex peeling is invariant to affine transformations, that is, we have
\begin{equation}\label{e.affine0}
h_{Y_m\cap P} = h_{a(Y_m\cap P)}\circ a
\end{equation}
for any nonsingular affine transformation $a$.  The first important consequence of this affine symmetry is the \emph{scale property}
\begin{equation}\label{e.affine1}
h_{ Y_m\cap P}(re_d)\sim h_{Y_{1}\cap P }(m^{2/(d+1)}re_d).
\end{equation}
To see this, we use \eref{affine0} with the affine transformation $a_m$ defined by 
\[a_m x = (m^{1(d+1)} x^d,m^{2/(d+1)}x_d),\]
noting that $a_mP=P$ and $a_mY_m\sim Y_1$.  Second, if $a$ is any affine transformation on $\R^{d}$ satisfying $ae_d=e_d$, then applying again  \eqref{e.affine0} we have
\begin{equation}\label{e.affine2}
h_{ Y_m\cap aP}(re_d)\sim h_{Y_{sm}\cap P }(re_d),
\end{equation}
where $s=|\det(Da)|>0$. Combining the two affine symmetries \eref{affine1} and \eref{affine2} with the cell problem \eref{cellproblem} we have
\begin{equation}\label{e.affine3}
h_{ Y_m\cap aP}(re_d)\approx \alpha(sm)^{2/(d+1)}r,
\end{equation}
provided $r \gg m^{-2/(d+1)}$. Thus, the solution of the cell problem \eref{cellproblem}, coupled with convenient affine invariances, allows us to solve a whole family of cell problems $h_{Y_m\cap aP}$ for Poisson clouds with arbitrary intensity $m$, and general parabolas $aP$. It is worth noting that that Gaussian curvature $\kappa_G$ of the parabla $aP$ at the origin  is given by $\kappa_G=s^{-2}$, and so $s^{2/(d+1)}= \kappa_G^{-1/(d+1)}$.

To solve the cell problem, i.e., prove \eref{cellproblem}, we first make the observation that the peeling of the parabola $Y_1\cap P$ has a spatial homogeneity that is not initially evident. Indeed, let $H=\{x\in \R^d \, : \, x_d>0\}$, define $\pi:P\to H$ by
\[\pi(x) = (x_1,x_2,\dots,x_{d-1},x_d-\tfrac12 |x^d|^2)\]
and set $s = h_{X_1\cap P}\circ \pi^{-1}$. Since $\pi$ is not affine, the function $s$ is not the depth function for a convex peeling. However, we can interpret $s$ as the depth function for another type of peeling that we call \emph{semiconvex peeling}. While the points removed in each layer of convex peeling are those that can be touched by half-spaces, the points removed by  semiconvex peeling are exactly those touched by downward facing parabolas $x_0 - P$, i.e., the images of halfspaces under the bijection $\pi$. See \fref{semiconvexpeeling} for an illustration of semiconvex peeling. Since $\pi(P)=H$ and $\pi(Y_1)\sim Poisson(1)$, the depth function $s$ describes the semiconvex peeling of a unit intensity Poisson point cloud above a halfspace, and we can immediately see that the process is distributionally invariant with respect to translations in $\R^{d-1}$. These additional symmetries of semiconvex peeling allow for a Martingale proof of convergence, which is given in \sref{semiconvex}.

The solution of the cell problem locally describes the convex depth function and can be used to derive the limiting PDE \eref{bvp}. Indeed, assume that $m^{-2/(d+1)}h_{X_m}$ is uniformly close to a smooth function $h$. Let $x_0\in U$ and assume, after making an orthogonal transformation if necessary, that $D h(x_0) = | Dh(x_0)|e_d$. Let $a$ be an affine transformation with $ae_d=e_d$ for which 
\[x_0 + aP \approx \{x\in B(x_0,r)\, : \, h(x)>h(x_0)\}.\]
In particular, defining $s:=|\det(Da)|$ we have $s^{2/(d+1)}= \kappa_G^{-1/(d+1)}$, where $\kappa_G$ is the Gaussian curvature of the level set of $h$ at $x_0$. Making a localization approximation $X_m \sim Y_{f(x_0)m}$ for small $r>0$ we have
\begin{align*}
|D h(x_0)|  &\approx  \frac{1}{r}(h_{x_0}+re_d) - h(x_0))\\
&\approx \frac{1}{rm^{2/(d+1)}}(h_{X_m}(x_0+re_d) - h_{X_m}(x_0))\\
&\approx \frac{1}{rm^{2/(d+1)}}h_{Y_{f(x_0)m}\cap aP}(re_d)\\
&\approx \alpha f(x_0)^{2/(d+1)}\kappa_G^{-1/(d+1)},
\end{align*}
where we used the solution of the cell problem \eref{affine3} in the last line. This can be compared with the geometric form of the continuum PDE given in \eref{altpde}.  We note that the arguments above are merely formal, and are meant to give some of the main ideas that motivated \tref{main} and our proof techniques. We make these arguments rigorous in the remainder of the paper. 
\nc

\subsection{Outline}

In \sref{semiconvex}, we study a related peeling process called semiconvex peeling.  This process has some additional symmetries that allow for a Martingale proof of convergence.  In \sref{viscosity}, we discuss the solution theory of the limiting PDE.  This is essentially standard, except for a folklore theorem on piece-wise smooth approximation of viscosity solutions.  In \sref{convex}, we use the convergence of semiconvex peeling to control local regions of the convex peeling and prove our main result.  This requires some delicate geometric arguments to translate between our two notions of peeling.

\subsection{Acknowledgments}

The first author was partially supported by NSF-DMS grant 1500829.  The second author was partially supported by the National Science Foundation and the Alfred P Sloan Foundation.

\section{Semiconvex Peeling}
\label{s.semiconvex}

\subsection{Definitions}

In this section we study an a priori different peeling problem that we call semiconvex peeling.  In \sref{convex}, we will see that this is the ``cell problem'' for convex peeling.  That is, it is the problem obtained by blow-up of the limiting convex peeling problem.  For now, we simply study a different problem.

Consider the parabolic region
\begin{equation*}
  P = \{ x \in \R^d : x_d > \tfrac12 |x^d|^2 \},
\end{equation*}
where
\begin{equation*}
  x^d = (x_1, ..., x_{d-1}).
\end{equation*}
Consider also the half space
\begin{equation*}
  H = \{ x \in \R^d : x_d > 0 \}.
\end{equation*}
We call a set $S \subseteq \overline H$ semiconvex if its complement is a union of sets of the form $H \cap (x - P)$ for $x \in H$.  Note that this implies $S$ is closed.  This definition is analogous to the complement of a convex set being a union of open half spaces.  The semiconvex hull of a set $X \subseteq H$ is defined to be
\begin{equation*}
  \semi(X) = H \setminus \bigcup \{ x - P : x \in H \mbox{ and } (x - P) \cap X = \emptyset \}.
\end{equation*}
The semiconvex peeling of a set $X \subseteq H$ is defined by
\begin{equation*}
  S_1(X) = \semi(X) \quad \mbox{and} \quad S_{n+1}(X) = \semi(X \cap \interior(S_n(X))).
\end{equation*}
The semiconvex height function of a set $X \subseteq H$ is defined to be
\begin{equation*}
  s_X = \sum_{n \geq 1} \id_{\interior(S_n(X))}.
\end{equation*}
Note that $s_X$ takes values in $\N \cup \{ \infty \}$ a priori.  Of course, when $X \subseteq H$ is locally finite, $s_X$ is everywhere finite.  See \fref{semiconvexpeeling} for an example.

Throughout the paper, $C$ and $c$ denote positive constants that may vary in each instance, but depend only on dimension. We always assume $C>1$ and $0 < c < 1$.
\begin{figure}
  \label{f.semiconvexpeeling}
  \input{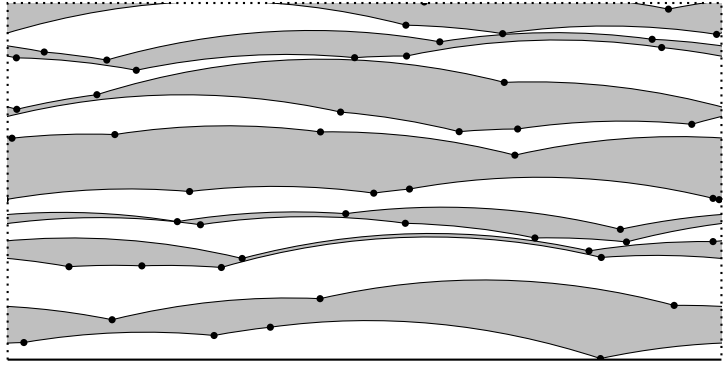}
  \caption{The semiconvex peeling of a $\Z^{d-1} \times \{ 0 \}$-periodic Poisson cloud $X \subseteq H$.  The shading indicates alternating semiconvex layers.}
\end{figure}

\subsection{Monotonicity}

Semiconvex peeling is monotone in the following sense.

\begin{lemma}
  \label{l.monotone}
  If $X \subseteq Y \subseteq H$, then $s_X \leq s_Y$.  More generally, if $\tilde S_n \subseteq H$ is a sequence of semiconvex sets that satisfy $\tilde S_{n+1} \subseteq \tilde S_n$ and $X \subseteq \bigcup_{n \geq 1} \partial \tilde S_n \cup \bigcap_{n \geq 1} \tilde S_n$, then $s_X \leq \sum_{n \geq 1} \id_{\tilde S_n}$.
\end{lemma}

\begin{proof}
  The second statement follows from the identity
  \begin{equation*}
    \semi(X) = \cap \{ K \mbox{ semiconvex} : X \subseteq K \subseteq H \}
  \end{equation*}
  and induction on $n$.  The first statement follows from the second.
\end{proof}

\subsection{Tail Bounds}

Like the convex height function, the semiconvex height function has a dynamic programming principle. \red  

\begin{proposition}\label{p.semidpp}
For all $x\in H$ we have
\begin{equation*}
  s_X(x) = \inf_{y \in x + \partial P} \sup_{z \in X \cap (y - P)} (\id_X(z) + s_X(z)),
\end{equation*}
where the empty supremum is interpreted as $0$.
\end{proposition}
\begin{proof}
Let $X_n = X\cap \interior(S_n(X))$ so that $S_{n+1}(X) = \semi(X_n)$, and set $X_0=X$. Note that $X_n\supset X_{n+1}$ and $s_X(x)= n$ if $x \in X_n\setminus X_{n+1}$ for all $n\geq 0$. Let $x\in H$ and set $n=s_X(x)$. Thus, $x\not\in \interior(S_{n+1}(X))$ and so there exists $\tilde{y}\in H$ such that $(\tilde{y}-P)\cap X_{n}=\emptyset$ and $x\in \overline{\tilde{y}-P}$. Let $t\geq 0$ such that $y:=\tilde{y}-te_d$ satisfies $x\in y - \partial P$, which is equivalent to $y\in x + \partial P$. Since $\tilde{y}-P \supset y-P$ we have $(y-P)\cap X_n=\emptyset$, and so $s_X(z)\leq n-1$ for all $z\in y-P$. Therefore
\[\inf_{y\in x+\partial P}\sup_{z\in X\cap(y-P)}(\id_X(z) + s_X(z)) \leq n=s_X(x).\]

If $n=0$, then the other inequality is trivial, so we may assume $n\geq 1$. Let $y\in x +\partial P$. Then we have $x \in \overline{y-P}$. If $(y-P)\cap X_{n-1}=\emptyset$, then we would have that $x\not\in \interior(S_{n}(X))$, which is a contradiction since $s_X(x)=n$. Therefore, there exists $z\in (y-P)\cap X_{n-1}$ and so 
\[\sup_{z\in X\cap(y-P)}(\id_X(z) + s_X(z)) \geq 1 + n-1 = n = s_X(x).\]
Since $y\in x+\partial P$ was arbitrary, the proof is complete.
\end{proof}
\nc
The dynamic programming principle given in \pref{semidpp} has the natural interpretation as a two-player zero-sum game.  We prove upper and lower tail bounds by constructing strategies in this game.  In both cases, we construct trees of disjoint regions in $H$, and trade the exponential tree growth against exponential bounds for the Poisson process.  Our upper bound strategy adapts an argument from \mycite{Dalal}.  Our lower bound strategy is new.

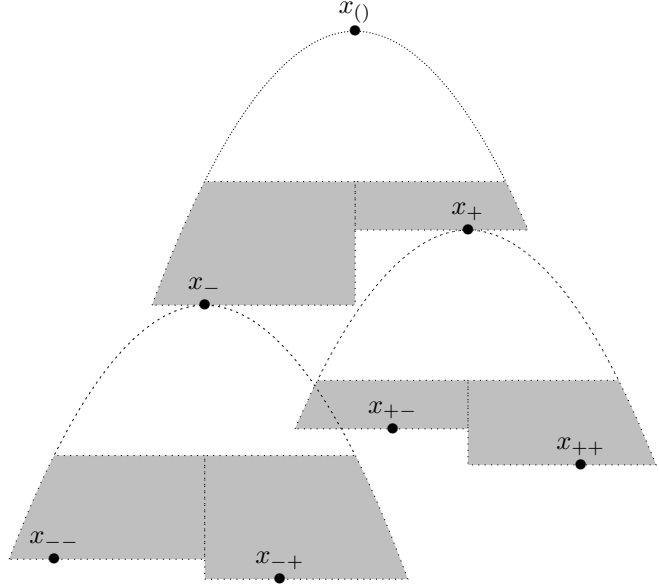
\begin{figure}[h]
  \begin{tikzpicture}
\draw[dotted,fill=lightgray] (-2.700000,-3.645000) parabola bend (0.000000,-0.000000) (2.300000,-2.645000) -- (0.000000,-2.645000) -- (0.000000,-2.000000) -- (2.000000,-2.000000) parabola bend (0.000000,-0.000000) (-2.000000,-2.000000) -- (0.000000,-2.000000) -- (0.000000,-3.645000) -- (-2.700000,-3.645000);
\draw[dotted,fill=lightgray] (-0.800000,-5.290000) parabola bend (1.500000,-2.645000) (4.000000,-5.770000) -- (1.500000,-5.770000) -- (1.500000,-4.645000) -- (3.500000,-4.645000) parabola bend (1.500000,-2.645000) (-0.500000,-4.645000) -- (1.500000,-4.645000) -- (1.500000,-5.290000) -- (-0.800000,-5.290000);
\draw[dotted,fill=lightgray] (-4.600000,-7.025000) parabola bend (-2.000000,-3.645000) (0.700000,-7.290000) -- (-2.000000,-7.290000) -- (-2.000000,-5.645000) -- (0.000000,-5.645000) parabola bend (-2.000000,-3.645000) (-4.000000,-5.645000) -- (-2.000000,-5.645000) -- (-2.000000,-7.025000) -- (-4.600000,-7.025000);
00) -- (-4.600000,-7.025000);
\draw (0.000000,-0.000000) node {\textbullet};
\draw (0.000000,0.250000) node {$x_{()}$};
\draw (1.500000,-2.645000) node {\textbullet};
\draw (1.500000,-2.395000) node {$x_{+}$};
\draw (-2.000000,-3.645000) node {\textbullet};
\draw (-2.000000,-3.395000) node {$x_{-}$};
\draw (3.000000,-5.770000) node {\textbullet};
\draw (3.000000,-5.520000) node {$x_{++}$};
\draw (0.500000,-5.290000) node {\textbullet};
\draw (0.500000,-5.040000) node {$x_{+-}$};
\draw (-1.000000,-7.290000) node {\textbullet};
\draw (-1.000000,-7.040000) node {$x_{-+}$};
\draw (-4.000000,-7.025000) node {\textbullet};
\draw (-4.000000,-6.775000) node {$x_{--}$};
\end{tikzpicture}
  \caption{A tree of points for the lower bound strategy.}
  \label{f.lower}
\end{figure}

For later applications, we need these bounds to be localized.  For $r > 0$, we define the cylinder
\begin{equation*}
  Q_r = \{ x \in \R^d : x_d \in (0,r) \mbox{ and } x_1^2 + \cdots + x_{d-1}^2 < r^2 \}
\end{equation*}
and its upper boundary
\begin{equation*}
  \partial^+ Q_r = \{ x \in \R^d : x_d = r \mbox{ and } x_1^2 + \cdots + x_{d-1}^2 < r^2 \}.
\end{equation*}

\begin{lemma}
  \label{l.lower}
  If $r \geq 1$ and $X \sim Poisson(\id_H)$, then
  \begin{equation*}
    \P[s_{X \cap Q_r}(r e_d) \leq c r] \leq e^{-c r}.
  \end{equation*}
\end{lemma}

\begin{proof}
  Step 1.  We may assume that $X = Y \cap H$, where $Y \sim Poisson(\id_{\R^d})$.  We build a tree by exploring the Poisson cloud $Y$ downward from $r e_d$.  We define a tree of points $x_u \in Y$ indexed by words $u$ in the alphabet $A = \{ -,+\}^{d-1}$.  Begin by setting $x_{()} = r e_d$.  For $v \in A$ and $t > 0$, we define the region
  \begin{equation*}
    P_{v,t} = \{ x \in P : 2 < x_d \leq 2 + t \mbox{ and } v_k x_k > 0 \mbox{ for } k = 1, ..., d-1 \}.
  \end{equation*}
  If $u \in A^k$, $v \in A$, and $x_u \in Y$ is already defined, we choose $t_{uv} > 0$ and $x_{uv} \in Y \cap (x_u - P_{v,t_{uv}})$ with $t_{uv} > 0$ as small as possible.  Using the Poisson law, we see that $x_{uv}$ and $t_{uv}$ exist almost surely.

  Our tree is chosen to that it provides a strategy for the maximizer in the semiconvex hull game.  Observe that, if $y \in x_u + \partial P$, then there is a $v \in A$ such that $x_{uv} \in y - P$.  Thus, if $x_{uv} \in H$ for all $v \in A$, the dynamic programming principle implies
  \begin{equation*}
    s_X(x_u) \geq 1 + \min_{v \in A} s_X(x_{uv}).
  \end{equation*}
  Thus, by induction, we see that
  \begin{equation*}
    s_X(r e_d) \geq \max \{ n : x_u \in H \mbox{ for all } u \in A^n \}.
  \end{equation*}
  Next, observe that the set $P_{v,t}$ was chosen in such a way that $x_{uv} \in x_u - Q_{2 + t_{uv}}$.  That is, if $x_u \in H$, then $x_u \in Q_r$.  This gives the localization
  \begin{equation*}
    s_{X \cap Q_r}(r e_d) \geq \max \{ n : x_u \in H \mbox{ for all } u \in A^n \}.
  \end{equation*}
  It remains to control the right-hand side.

  Step 2.
  Fix a word $u = v_1 \cdots v_n \in A^n$ and consider its initial segments $u_k = v_1 \cdots v_k$.  Observe that the sets
  \begin{equation*}
    P_k = x_{u_{k-1}} - P_{v_k,t_{u_k}}
  \end{equation*}
  are disjoint.  Thus, by the Poisson law, the sequence of random heights $t_{u_k}$ are independent.  Since the sections have volume satisfying
  \begin{equation*}
    |P_{v,t}| \geq c t,
  \end{equation*}
  the Poisson law of $Y$ gives
  \begin{equation*}
    \E[e^{t_{u_k}}] \leq e^C
  \end{equation*}
  and, by independence,
  \begin{equation*}
    \E[e^{(x_u)_d - r}] = \E[e^{t_{u_1} + \cdots + t_{u_n}}] = \prod_{k=1}^n \E[e^{t_{u_k}}] \leq e^{Cn}.
  \end{equation*}
  Applying Chebyshev's inequality and summing over words $u \in A^n$, it follows that
  \begin{equation*}
    \P\left[\max_{u \in A^n} (r - (x_u)_d) \geq r \right] \leq e^{C n - r}.
  \end{equation*}
  Combining this with the previous step, we see that
  \begin{equation*}
    \P[s_{X \cap Q_r}(r e_d) <  n] \leq \P[\max_{u \in A^n} (r - (x_u)_d) \geq r] \leq e^{C n - r}.
  \end{equation*}
  Setting $n = \lceil c r \rceil$ yields the lemma.
\end{proof}

\begin{figure}[h]
  \begin{tikzpicture}[scale=0.07\textwidth,x=1,y=1]
\begin{scope}
\clip (-5,0) rectangle (5,-5);
\draw[lightgray] (-7.500000,-6.125000) parabola bend (-6.000000,-5.000000) (-4.500000,-6.125000);
\draw[lightgray] (-7.500000,-5.125000) parabola bend (-6.000000,-4.000000) (-4.500000,-5.125000);
\draw[lightgray] (-7.500000,-4.125000) parabola bend (-6.000000,-3.000000) (-4.500000,-4.125000);
\draw[lightgray] (-7.500000,-3.125000) parabola bend (-6.000000,-2.000000) (-4.500000,-3.125000);
\draw[lightgray] (-7.500000,-2.125000) parabola bend (-6.000000,-1.000000) (-4.500000,-2.125000);
\draw[lightgray] (-7.500000,-1.125000) parabola bend (-6.000000,0.000000) (-4.500000,-1.125000);
\draw[lightgray] (-6.500000,-6.125000) parabola bend (-5.000000,-5.000000) (-3.500000,-6.125000);
\draw[lightgray] (-6.500000,-5.125000) parabola bend (-5.000000,-4.000000) (-3.500000,-5.125000);
\draw[lightgray] (-6.500000,-4.125000) parabola bend (-5.000000,-3.000000) (-3.500000,-4.125000);
\draw[lightgray] (-6.500000,-3.125000) parabola bend (-5.000000,-2.000000) (-3.500000,-3.125000);
\draw[lightgray] (-6.500000,-2.125000) parabola bend (-5.000000,-1.000000) (-3.500000,-2.125000);
\draw[lightgray] (-6.500000,-1.125000) parabola bend (-5.000000,0.000000) (-3.500000,-1.125000);
\draw[black] (-5.500000,-6.125000) parabola bend (-4.000000,-5.000000) (-2.500000,-6.125000);
\draw[lightgray] (-5.500000,-5.125000) parabola bend (-4.000000,-4.000000) (-2.500000,-5.125000);
\draw[lightgray] (-5.500000,-4.125000) parabola bend (-4.000000,-3.000000) (-2.500000,-4.125000);
\draw[lightgray] (-5.500000,-3.125000) parabola bend (-4.000000,-2.000000) (-2.500000,-3.125000);
\draw[lightgray] (-5.500000,-2.125000) parabola bend (-4.000000,-1.000000) (-2.500000,-2.125000);
\draw[lightgray] (-5.500000,-1.125000) parabola bend (-4.000000,0.000000) (-2.500000,-1.125000);
\draw[black] (-4.500000,-6.125000) parabola bend (-3.000000,-5.000000) (-1.500000,-6.125000);
\draw[black] (-4.500000,-5.125000) parabola bend (-3.000000,-4.000000) (-1.500000,-5.125000);
\draw[lightgray] (-4.500000,-4.125000) parabola bend (-3.000000,-3.000000) (-1.500000,-4.125000);
\draw[lightgray] (-4.500000,-3.125000) parabola bend (-3.000000,-2.000000) (-1.500000,-3.125000);
\draw[lightgray] (-4.500000,-2.125000) parabola bend (-3.000000,-1.000000) (-1.500000,-2.125000);
\draw[lightgray] (-4.500000,-1.125000) parabola bend (-3.000000,0.000000) (-1.500000,-1.125000);
\draw[black] (-3.500000,-6.125000) parabola bend (-2.000000,-5.000000) (-0.500000,-6.125000);
\draw[black] (-3.500000,-5.125000) parabola bend (-2.000000,-4.000000) (-0.500000,-5.125000);
\draw[black] (-3.500000,-4.125000) parabola bend (-2.000000,-3.000000) (-0.500000,-4.125000);
\draw[lightgray] (-3.500000,-3.125000) parabola bend (-2.000000,-2.000000) (-0.500000,-3.125000);
\draw[lightgray] (-3.500000,-2.125000) parabola bend (-2.000000,-1.000000) (-0.500000,-2.125000);
\draw[lightgray] (-3.500000,-1.125000) parabola bend (-2.000000,0.000000) (-0.500000,-1.125000);
\draw[black] (-2.500000,-6.125000) parabola bend (-1.000000,-5.000000) (0.500000,-6.125000);
\draw[black] (-2.500000,-5.125000) parabola bend (-1.000000,-4.000000) (0.500000,-5.125000);
\draw[black] (-2.500000,-4.125000) parabola bend (-1.000000,-3.000000) (0.500000,-4.125000);
\draw[black] (-2.500000,-3.125000) parabola bend (-1.000000,-2.000000) (0.500000,-3.125000);
\draw[lightgray] (-2.500000,-2.125000) parabola bend (-1.000000,-1.000000) (0.500000,-2.125000);
\draw[lightgray] (-2.500000,-1.125000) parabola bend (-1.000000,0.000000) (0.500000,-1.125000);
\draw[black] (-1.500000,-6.125000) parabola bend (0.000000,-5.000000) (1.500000,-6.125000);
\draw[black] (-1.500000,-5.125000) parabola bend (0.000000,-4.000000) (1.500000,-5.125000);
\draw[black] (-1.500000,-4.125000) parabola bend (0.000000,-3.000000) (1.500000,-4.125000);
\draw[black] (-1.500000,-3.125000) parabola bend (0.000000,-2.000000) (1.500000,-3.125000);
\draw[black] (-1.500000,-2.125000) parabola bend (0.000000,-1.000000) (1.500000,-2.125000);
\draw[lightgray] (-1.500000,-1.125000) parabola bend (0.000000,0.000000) (1.500000,-1.125000);
\draw[black] (-0.500000,-6.125000) parabola bend (1.000000,-5.000000) (2.500000,-6.125000);
\draw[black] (-0.500000,-5.125000) parabola bend (1.000000,-4.000000) (2.500000,-5.125000);
\draw[black] (-0.500000,-4.125000) parabola bend (1.000000,-3.000000) (2.500000,-4.125000);
\draw[black] (-0.500000,-3.125000) parabola bend (1.000000,-2.000000) (2.500000,-3.125000);
\draw[lightgray] (-0.500000,-2.125000) parabola bend (1.000000,-1.000000) (2.500000,-2.125000);
\draw[lightgray] (-0.500000,-1.125000) parabola bend (1.000000,0.000000) (2.500000,-1.125000);
\draw[black] (0.500000,-6.125000) parabola bend (2.000000,-5.000000) (3.500000,-6.125000);
\draw[black] (0.500000,-5.125000) parabola bend (2.000000,-4.000000) (3.500000,-5.125000);
\draw[black] (0.500000,-4.125000) parabola bend (2.000000,-3.000000) (3.500000,-4.125000);
\draw[lightgray] (0.500000,-3.125000) parabola bend (2.000000,-2.000000) (3.500000,-3.125000);
\draw[lightgray] (0.500000,-2.125000) parabola bend (2.000000,-1.000000) (3.500000,-2.125000);
\draw[lightgray] (0.500000,-1.125000) parabola bend (2.000000,0.000000) (3.500000,-1.125000);
\draw[black] (1.500000,-6.125000) parabola bend (3.000000,-5.000000) (4.500000,-6.125000);
\draw[black] (1.500000,-5.125000) parabola bend (3.000000,-4.000000) (4.500000,-5.125000);
\draw[lightgray] (1.500000,-4.125000) parabola bend (3.000000,-3.000000) (4.500000,-4.125000);
\draw[lightgray] (1.500000,-3.125000) parabola bend (3.000000,-2.000000) (4.500000,-3.125000);
\draw[lightgray] (1.500000,-2.125000) parabola bend (3.000000,-1.000000) (4.500000,-2.125000);
\draw[lightgray] (1.500000,-1.125000) parabola bend (3.000000,0.000000) (4.500000,-1.125000);
\draw[black] (2.500000,-6.125000) parabola bend (4.000000,-5.000000) (5.500000,-6.125000);
\draw[lightgray] (2.500000,-5.125000) parabola bend (4.000000,-4.000000) (5.500000,-5.125000);
\draw[lightgray] (2.500000,-4.125000) parabola bend (4.000000,-3.000000) (5.500000,-4.125000);
\draw[lightgray] (2.500000,-3.125000) parabola bend (4.000000,-2.000000) (5.500000,-3.125000);
\draw[lightgray] (2.500000,-2.125000) parabola bend (4.000000,-1.000000) (5.500000,-2.125000);
\draw[lightgray] (2.500000,-1.125000) parabola bend (4.000000,0.000000) (5.500000,-1.125000);
\draw[lightgray] (3.500000,-6.125000) parabola bend (5.000000,-5.000000) (6.500000,-6.125000);
\draw[lightgray] (3.500000,-5.125000) parabola bend (5.000000,-4.000000) (6.500000,-5.125000);
\draw[lightgray] (3.500000,-4.125000) parabola bend (5.000000,-3.000000) (6.500000,-4.125000);
\draw[lightgray] (3.500000,-3.125000) parabola bend (5.000000,-2.000000) (6.500000,-3.125000);
\draw[lightgray] (3.500000,-2.125000) parabola bend (5.000000,-1.000000) (6.500000,-2.125000);
\draw[lightgray] (3.500000,-1.125000) parabola bend (5.000000,0.000000) (6.500000,-1.125000);
\draw[lightgray] (4.500000,-6.125000) parabola bend (6.000000,-5.000000) (7.500000,-6.125000);
\draw[lightgray] (4.500000,-5.125000) parabola bend (6.000000,-4.000000) (7.500000,-5.125000);
\draw[lightgray] (4.500000,-4.125000) parabola bend (6.000000,-3.000000) (7.500000,-4.125000);
\draw[lightgray] (4.500000,-3.125000) parabola bend (6.000000,-2.000000) (7.500000,-3.125000);
\draw[lightgray] (4.500000,-2.125000) parabola bend (6.000000,-1.000000) (7.500000,-2.125000);
\draw[lightgray] (4.500000,-1.125000) parabola bend (6.000000,0.000000) (7.500000,-1.125000);
\end{scope}
\draw (-5,-5) -- (5,-5);
\draw (0,-1) node {\textbullet};
\draw (0,-0.75) node {$x$};
\draw (-1,-2) node {\textbullet};
\draw (0,-2) node {\textbullet};
\draw (1,-2) node {\textbullet};
\draw (0,-1.75) node {$S(x)$};
\draw (0,-5.25) node {$\partial H$};
\end{tikzpicture}
  \caption{A tree of parabolic caps from the upper bound strategy.}
  \label{f.upper}
\end{figure}
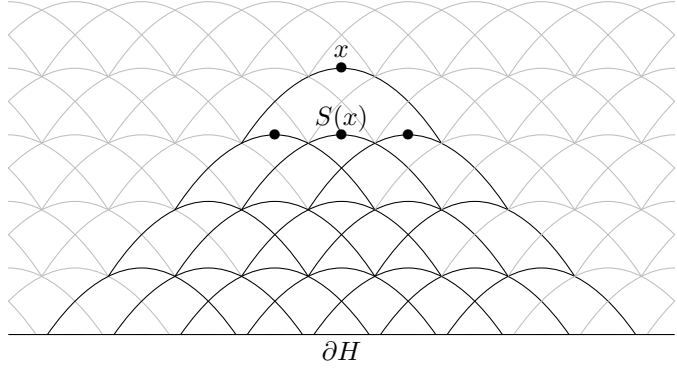

To prove an upper bound, we employ the only canonical strategy for the minimizer: choosing $y = x$.  To estimate the performance of this strategy, we build a tree of parabolic caps.  See \fref{upper} for a picture in dimension $d = 2$.  This is a straightforward adaptation of a lemma from \mycite{Dalal}.

\begin{lemma}
  \label{l.upper}
  If $r, t \geq 1$ and $X \sim Poisson(\id_H)$, then
  \begin{equation*}
    \P[s_{X \cup (H \setminus Q_r)}(r e_d) \geq C r t] \leq e^{- r t}.
  \end{equation*}
\end{lemma}

\begin{proof}
  Step 1.
  For $x \in \Z^d$, let
  \begin{equation*}
    S(x) = \{ y \in \Z^d : y_d = x_d - 2 \mbox{ and } |y^d - x^d|_\infty \leq 2 \}
  \end{equation*}
  and observe that the set
  \begin{equation*}
    \Omega_x = (x - P) \setminus \bigcup_{y \in S(x)} (y - P)
  \end{equation*}
  satisfies
  \begin{equation*}
    \Omega_x \subseteq x - Q_4.
  \end{equation*}
  For a picture of these parabolic caps in $d = 2$, see \fref{upper}.  As in \mycite{Dalal}, the dynamic programming principle implies that
  \begin{equation*}
    s_X(x) \leq \# (X \cap \Omega_x) + \max_{y \in S(x)} s_X(y)
  \end{equation*}
  holds for all $x \in \Z^d$.

  Step 2.  For $x \in \Z^d$ with $x_d = 2 n > 0$, let
  \begin{equation*}
    T(x) = \{ \mathbf y \in (\Z^d)^{n+1} : \mathbf y_1 = x \mbox{ and } \mathbf y_{k+1} \in S(\mathbf y_k).
  \end{equation*}
  Since $s_X(\mathbf y_{n+1}) = 0$, the previous step implies that
  \begin{equation*}
    s_X(x) \leq \max_{\mathbf y \in T(x)} \sum_{k = 1}^n \#(X \cap \Omega_{\mathbf y_k}) = \max_{\mathbf y \in T(x)} \# \left( X \cap \bigcup_{k=1}^n \Omega_{\mathbf y_k} \right).
  \end{equation*}
  Since
  \begin{equation*}
    \# T(x) \leq e^{C x_d} \quad \mbox{and} \quad \left| \bigcup_{k=1}^n \Omega_{\mathbf y_k} \right| \leq C x_d,
  \end{equation*}
  the Poisson law together with a union bound gives
  \begin{equation*}
    \P[s_X(x) \geq t] \leq e^{C x_d - t}.
  \end{equation*}
  Replacing $t$ by $C x_d t$ yields the lemma for $r > 0$ an even integer.  Obtain the remaining $r$ by observing that $r \mapsto s_X(r e_d)$ is non-decreasing.
\end{proof}

\subsection{Localization}

Using the tail bounds, which are already localized, we obtain full localization of semiconvex peeling.  The essential idea is that, if the structure of $X \setminus Q_{\lalpha r}$ is affecting the value of $s_X(r e_d)$, then some point on $\partial^+ Q_{\lalpha r}$ has height less than some point on $\partial^+ Q_r$.  If $\lalpha > 1$ is large, then the tail bounds imply this is unlikely.  This situation is depicted in \fref{local}.

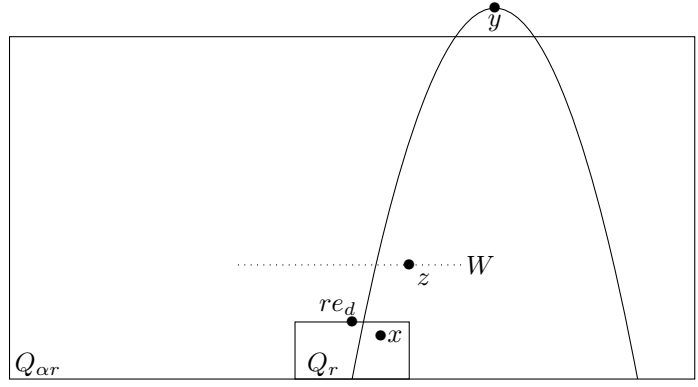
\begin{figure}[h]
  \begin{tikzpicture}[scale=.03\textwidth,x=1,y=1]

  \draw (-12,0) rectangle (12,12);
  \draw (-11,0.5) node {$Q_{\alpha r}$};
  
  \draw (-2,0) rectangle (2,2);
  \draw (-1,0.5) node {$Q_r$};
  
  \draw (0,2) node {$\bullet$};
  \draw (-0.5,2.5) node {$r e_d$};
  
  \draw (1,1.5) node {$\bullet$};
  \draw (1.5,1.5) node {$x$};
  
  \draw (0,0) parabola bend (5,13) (10,0);
  \draw (5,13) node {$\bullet$};
  \draw (5,12.5) node {$y$};
  
  \draw[dotted] (-4,4) -- (4,4);
  \draw (2,4) node {$\bullet$};
  \draw (2.5,3.5) node {$z$};
  \draw (4.5,4) node {$W$};
  
\end{tikzpicture}
  \caption{A schematic of localization failure.}
  \label{f.local}
\end{figure}

\begin{lemma}
  \label{l.local}
  There is an $\lalpha \geq 1$ such that, if $r \geq 1$, then
  \begin{equation*}
    \P[s_{X \cap Q_{\lalpha r}}(r e_d) \neq s_{X \cup (H \setminus Q_{\lalpha r})}(r e_d)] \leq e^{- r}.
  \end{equation*}
\end{lemma}

\begin{proof}
  Suppose $s_{X \cap Q_{\lalpha r}}(r e_d) < s_{X \cup (H \setminus Q_{\lalpha r})}(r e_d)$ for some $\lalpha \geq 1$ large and to be determined.  Write $X_1 = X \cap Q_{\lalpha r}$ and $X_2 = X \cup (H \setminus Q_{\lalpha r})$.  By hypothesis, there is a least $n \leq s_{X_1}(r e_d)$ such that $S_n(X_1) \cap Q_r \neq S_n(X_2) \cap Q_r$.

  By monotonicity, $S_n(X_1) \subseteq S_n(X_2)$.  Thus, we can choose a point $x \in Q_r \cap S_n(X_2) \setminus S_n(X_1)$.  By the definition of semiconvex peeling, there is a $y \in H$ such that $x \in (y - P)$, $(y - P) \cap X_1 \cap \interior(S_{n-1}(X_1)) = \emptyset$, and $(y - P) \cap X_2 \cap \interior(S_{n-1}(X_2)) \neq \emptyset$.  Since $Q_r \cap S_{n-1}(X_1) = Q_r \cap S_{n-1}(X_2)$ and $X_1 \cap Q_{\lalpha r} = X_2 \cap Q_{\lalpha r}$, we have $(y - P) \setminus Q_{\lalpha r} \neq \emptyset$.  Making $\lalpha \geq 1$, there must be a $z \in (y - P) \cap Q_{\lalpha r}$ with
  \begin{equation*}
    z \in W = (\Z^{d-1} \times \R) \cap \partial^+ Q_{\lalpha r / 3}.
  \end{equation*}
  See \fref{local} for a schematic of our situation.  Observe that
  \begin{equation*}
    s_{X \cap (z^d + Q_{\lalpha r /3})}(z) \leq s_{X \cap Q_{\lalpha r}}(z) \leq s_{X \cap Q_{\lalpha r}}(x) \leq s_{X \cup (H \setminus Q_r)}(r e_d).
  \end{equation*}
  Using the tail bounds in \lref{lower} and \lref{upper}, the probability this happens for fixed $z \in W$ is at most $\exp(- c \lalpha s)$.  Since $\# W \leq C \lalpha^{d-1} r^{d-1}$, the lemma follows by a union bound.
\end{proof}

\subsection{Concentration}

The localization of semiconvex peeling allows us to periodize our problem.  That is, we are able to replace the half space $H$ by a cylinder over a torus.  The primary advantage of this is that the semiconvex peels $S_n(X)$ on the cylinder are a priori Lipschitz graphs over compact sets.  This additional regularity allows us run a Martingale argument.  We follow \mycite{Armstrong-Cardaliaguet}, replacing their ingredients with our analogues.

For $L \geq 1$, consider the $(L \Z^{d-1} \times \{ 0 \})$-periodization
\begin{equation*}
  X^L = (X \cap (-\tfrac12 L,\tfrac12 L)^{d-1} \times (0,\infty)) + L (\Z^{d-1} \times \{ 0 \}).
\end{equation*}
\red The point cloud $X^L$ is the $L\Z^{d-1}$-periodic extension of $X\cap (-\tfrac12 L,\tfrac12 L)^{d-1}\times (0,\infty)$. \nc Localization immediately yields the following.

\begin{lemma}
  \label{l.periodize}
  If $r \geq 1$ and $L \geq C r$, then
  \begin{equation*}
    \P[s_X(r e_d) \neq s_{X^L}(r e_d) ] \leq e^{-r}.
  \end{equation*}
\end{lemma}

\begin{proof}
  This is immediate from \lref{local}.
\end{proof}

Our present goal is to prove the following fluctuation bound.

\begin{lemma}
  \label{l.fluct}
  For $r \geq t \geq C$ and $L \geq C r$,
  \begin{equation*}
    \P[|s_{X^L}(r e_d) - \E[s_{X^L}(r e_d)]| \geq (\log L)^2 (\log r) r^{1/2} t] \leq C e^{-c t^{2/3}}.
  \end{equation*}
\end{lemma}

For the remainder of this subsection, we write $X$ in place of $X^L$.

There is a natural filtration associated to the semiconvex peeling of $X$.  Let $F_n$ be the $\sigma$-algebra generated by $S_n(X)$ and $X \setminus \interior(S_n(X))$.  For $r \geq 1$, we study the Martingale
\begin{equation*}
  Y_n = \E[s_X(r e_d) | F_n].
\end{equation*}
We prove concentration by obtaining bounds on the increments.

We measure the increments using a swapping trick.  Let $\tilde X$ be an independent copy of $X$ and define the swapped point clouds
\begin{equation*}
  X_n = (X \setminus \interior(S_n(X))) \cup (\tilde X \cap \interior(S_n(X))).
\end{equation*}
The swapped point cloud $X_n$ is obtained by switching from $X$ to $\tilde X$ after $n$ peels.

The key observation is that, since $X$ and $\tilde X$ are independent, the swapped cloud $X_n$ has the same law as $X$.  Moreover, we have $\E[X_n|F_n] = \E[X|F_n]$, and thus
\begin{equation*}
  \begin{aligned}
    Y_{n+1} - Y_n
    & = \E[s_X(r e_d) | F_{n+1}] - \E[s_X(r e_d) | F_n] \\
    & = \E[s_{X_{n+1}}(r e_d) | F_{n+1}] - \E[\E[s_{X_n}(r e_d) | F_{n+1}]|F_{n}] \\
    & = \E[s_{X_{n+1}}(r e_d) - s_{X_n}(r e_d)| F_{n+1}],
  \end{aligned}
\end{equation*}
\red where in the last line we use that $s_{X_n}$ and $F_{n+1}$ are conditionally independent given $F_n$. \nc To understand the increment $Y_{n+1} - Y_n$, it suffices to relate the peelings of the point clouds $X_n$ and $X_{n+1}$.  This is tractable because the point clouds $X_n$ and $X_{n+1}$ differ only in the strip $\interior(S_n(X)) \setminus \interior(S_{n+1}(X))$.  We see below that the height of this strip controls in the increment.

The upper bound on the increments is easy, since it corresponds to the case where the strip has zero height and the increment is $1$.

\begin{lemma}
  Almost surely, $Y_{n+1} - Y_n \leq 1$.
\end{lemma}

\begin{proof}
  From the definitions, we obtain $X_{n+1} \subseteq X_n \cup \partial S_{n+1}(X)$.  Thus, the sets
  \begin{equation*}
    S_m' = \begin{cases}
      S_m(X) & \mbox{if } m \leq n + 1 \\
      S_{n+1}(X) \cap S_{m-1}(X_n) & \mbox{if } m \geq n + 2
    \end{cases}
  \end{equation*}
  are decreasing, semiconvex, and satisfy $X_{n+1} \subseteq \cup_{m \geq 1} \partial S'_m$.  Using \lref{monotone}, we obtain $s_{X_{n+1}} \leq \sum_m \id_{S'_m} \leq s_{X_n} + \id_{S_{n+1}(X)} \leq s_{X_n} + 1$.  Conclude by the swapping trick described above.
\end{proof}

The lower bound on the increments is harder, since the depth of the strip is a priori unbounded.  It is here that we use the simplifications afforded by the periodization.  The $(L \Z^{d-1} \times \{ 0 \})$-periodicity implies that the boundary $\partial S_n(X)$ is the graph of a $L \Z^{d-1}$-periodic and $C L$-Lipschitz function over the hyperplane $\partial H$.

\begin{lemma}
  Almost surely, $\P[Y_n - Y_{n+1} \geq C (\log L)^2 t | F_n] \leq \exp(-t)$.
\end{lemma}

\begin{proof}
  The proof is divided into three steps.  First, we show that, if the increment is large, then there must be many points of $\tilde X$ contained in a parabolic sector of the strip $S_n(X) \setminus S_{n+1}(X)$.  Second, we show that this is exponentially unlikely unless some parabolic sector of the strip has large volume.  Third, we show that parabolic sectors of the strip with large volume is exponentially unlikely.  The Lipschitz regularity of the peels allows us to consider only polynomially many parabolic sectors.
  
  Step 1.  We prove that, almost surely, $Y_n - Y_{n+1} \leq Z_n$, where
  \begin{equation*}
    Z_n = \sup \{ \# (\tilde X \cap (y - P) \cap S_n(x)) : y \in S_n(x) \setminus S_{n+1}(X)\}.
  \end{equation*}
  By the swapping trick, we must show
  \begin{equation*}
    \E[s_{X_n}(r e_d)|F_n] \leq Z_n + \E[s_{X_{n+1}}(r e_d)| F_{n+1}].
  \end{equation*}
  We add some peels to $X_{n+1}$ to obtain a peeling of $X_n$ and then conclude by monotonicity.  We need to add peels to absorb the points
  \begin{equation*}
    X_n \setminus X_{n+1} = \tilde X \cap \interior(S_n(X)) \setminus S_{n+1}(X).
  \end{equation*}
  Consider
  \begin{equation*}
    S'_m = S_m(S_{n+1}(X) \cup (\interior(S_n(X)) \cap \tilde X)).
  \end{equation*}
  Observe that, if $y \in \interior(S'_m) \setminus S_{n+1}(X)$, then $(y - P) \cap \tilde X \cap \partial S'_l \neq \emptyset$ for $1 \leq l \leq m$.  It follows that
  \begin{equation*}
    S'_{Z_n} = S_{n+1}(X).
  \end{equation*}
  Consider the following nested semiconvex sets:
  \begin{equation*}
    S''_m = 
    \begin{cases}
      S_m(X) & \mbox{if } m \leq n \\
      S'_{m-n} & \mbox{if } n < m < n + Z_n \\
      S_{m-n-Z_n+1}(X_{n+1}) & \mbox{if } m \geq n + Z_n.
    \end{cases}
  \end{equation*}
  That is, we insert $S'_1, ..., S'_{Z_n-1}$ in between the $S_n(X_{n+1})$ and $S_{n+1}(X_{n+1})$.  By construction, we have $X_n \subseteq \cup_{m \geq 1} \partial S''_m$.  The monotonicity from \lref{monotone} implies that $s_{X_{n+1}} \leq \sum \id_{K''_m} \leq s_{X_n} + Z_n$.  Since $S'_1, ..., S'_{Z_n-1}$ are $F_{n+1}$-measurable, we have $Y_n - Y_{n+1} \leq Z_n$.

  Step 2.  We prove that, almost surely, $Z_n \leq C (\log L) W_n$, where
  \begin{equation*}
    W_n = 1 + \sup \{ |(y-P) \cap S_n(X)| : y \in S_n(X) \setminus \interior{S_{n+1}(X)} \}.
  \end{equation*}
  That is, $W_n$ is $1$ plus the largest volume of a parabolic section of the strip $S_n(X) \setminus S_{n+1}(X)$.  Note that $W_n$ is unchanged if we restrict $y$ to lie in $\partial S_{n+1}(X)$.  Since $\partial S_{n+1}(X)$ is a $CL$-Lipschitz graph over the set $(\R / L \Z)^{d-1} \times \{ 0 \}$, it has area $C L^{d-1}$.  We may therefore select, in an $F_{n+1}$-measurable way, points $y_1, ..., y_N$ with $N \leq C L^C$, such that, for any $y \in \partial S_{n+1}(X)$, there is a $y_k$ with $y - P \subseteq y_k - P$ and $|(y_k - P) \cap S_n(X)| \leq W_n$.  Using the Poisson law and a union bound, we see that
  \begin{equation*}
    \begin{aligned}
    \P[Z_n \geq t W_n | F_{n+1}]
    & \leq \P[\max_k \# (\tilde X \cap (y_k - P) \cap S_n(X)) \geq t W_n | F_{n+1}] \\
    & \leq C L^C \exp( - c t) \\
    & \leq C \exp ( - c t - C \log L).
    \end{aligned}
  \end{equation*}
  In particular, $\E[Z_n | F_{n+1}] \leq C (\log L) W_n$.

  Step 3.  We prove that, almost surely, for $t \geq 1$, $\P[W_n \geq C (\log L) t | F_n ] \leq \exp(-t)$.  When combined with steps 1 and 2, this gives the lemma.  Fix $t \geq 1$.   Using the $C L$-Lipschitz regularity of $\partial S_n(X)$, we can choose, in an $F_n$ measurable way, $C L^C$ many points $y_k$ such that, if $y \in \partial S_n(X)$ and $| (y - P) \cap S_n(X)| \geq t$, there is a $y_k \in y - P$ such that $|(y_k - P) \cap S_n(X)| \geq \tfrac12 t$.
  
  In the event that $W_n \geq t$, there is a $y_k \in S_n(X) \setminus S_{n+1}(X)$.  In particular, there is a $y_k$ such that $\tilde X \cap (y_k - P) \cap \interior(S_n(X)) = \emptyset$.  The Poisson law and a union bound implies $\P[W_n \geq t|F_n] \leq C L^C \exp(-c t) \leq C \exp( - c t - C\log L).$
\end{proof}

We interpolate the increment bounds with Azuma's inequality to obtain concentration.  This is standard, but we include a proof for completeness.

\begin{lemma}
  \label{l.azuma}
  For $t, n \geq e$, $\P[|Y_n - Y_0| \geq C (\log L)^2 (\log n) n^{1/2} t] \leq \exp(-t^{2/3})$.
\end{lemma}

\begin{proof}
  For $\beta \geq 1$, define the truncated increments
  \begin{equation*}
    Z_n = (Y_{n+1} - Y_n) \id_{|Y_{n+1} - Y_n| \leq \beta}
  \end{equation*}
  and observe that
  \begin{multline*}
    \P[|Y_n - Y_0| \geq \lalpha] \leq \sum_{k=0}^{n-1} \P[|Y_{k+1} - Y_k| > \beta] \\ + \P\left[\left|\sum_{k=0}^{n-1} Z_k - \E[Z_k|F_k]\right| \geq \lalpha - \left| \sum_{k=0}^{n-1} \E[Z_k|F_k] \right| \right].
  \end{multline*}
  If $\beta = C (\log L)^2 (\log n) t^{2/3}$, then the increment bounds imply, almost surely,
  \begin{equation*}
    \P[|Y_{k+1} - Y_k| > \beta] \leq \exp(- (\log n) t^{2/3} )
  \end{equation*}
  and
  \begin{equation*}
    \E[Z_k | F_k] \leq \exp(- (\log n) t^{2/3} ).
  \end{equation*}
  Azuma's inequality implies
  \begin{equation*}
    \P[|Y_n - Y_0| \geq \lalpha] \leq n e^{-(\log n) t^{2/3}} + \exp(-\tfrac12 n^{-1} \beta^{-2} (\lalpha - n e^{-(\log n) t^{2/3}})^2).
  \end{equation*}
  Setting $\lalpha = C (\log L)^2 (\log n) n^{1/2} t$ and assuming $t \geq C$, this becomes
  \begin{equation*}
    \P[|Y_n - Y_0] \geq C (\log L)^2 (\log n) n^{1/2} t] \leq \exp(-t^{2/3}).
  \end{equation*}
  Making the constant larger, we may assume $t \geq 1$.
\end{proof}

We now adapt the above estimate to prove the main fluctuation bound.

\begin{proof}[Proof of \lref{fluct}]
  Since $L \geq Cr$, the upper tail bound \lref{upper} implies
  \begin{equation*}
    \P[s_X(r e_d) \neq Y_{rt}] \leq C e^{-c r t}.
  \end{equation*}
  On the other hand, \lref{azuma} implies
  \begin{equation*}
    \P[|Y_{rt} - \E[s_X(r e_d)]| \geq C (\log L)^2 (\log r) r^{1/2} t] \leq e^{-t^{2/3}}.
  \end{equation*}
  Assuming $r \geq t \geq 1$, these combine to give the lemma.
\end{proof}

\subsection{Convergence}

Note that in this subsection we return to the general semiconvex peeling problem, where $X \sim Poisson(\id_H)$ has not been periodized.  In light of the fluctuation bounds in \lref{fluct}, all that remains is to control the expectation of $s_X(r e_d)$.  This is achieved by proving approximate additivity.

\red
\begin{lemma}
  \label{l.additivity}
  For $r \geq C$ and $t>0$
  \[\E[|s_X( (r+t) e_d) - s_X(r e_d) - s_X(te_d)|] \leq C (\log r)^3 r^{1/2}.\]
\end{lemma}

\begin{proof}
  By \lref{periodize} and \lref{upper}, we may assume that $X = X^L$ is $(L \Z^{d-1} \times \{ 0 \})$-periodic for some $L = C r^2$.  Consider the quantities
  \begin{equation*}
    n^- = \inf_{\R^{d-1} \times \{ r \}} s_X
    \quad \mbox{and} \quad
    n^+ = \sup_{\R^{d-1} \times \{ r \}} s_X.
  \end{equation*}
  Using the a priori $C r^2$-Lipschitz regularity of $\partial S_n(X)$ and the fluctuation bounds in \lref{fluct}, we obtain
  \begin{equation*}
    0 \leq \E[n^+ - n^-] \leq C (\log r)^3 r^{1/2}.
  \end{equation*}
  Note that $S_{n^+}(X) \subseteq \R^{d-1} \times (r,\infty) \subseteq S_{n^-}(X)$.  We use this to define two peelings:
  \begin{equation*}
    S^-_n = \begin{cases}
      S_n(X) & \mbox{if } n \leq n^- \\
      S_{n - n^-}(X \cap \R^{d-1}\times (r,\infty)) & \mbox{if } n > n^-
    \end{cases}
  \end{equation*}
  and      
  \begin{equation*}
    S^+_n = \begin{cases}
      S_n(X) & \mbox{if } n \leq n^+ \\
      S_{n^+}(X) \cap S_{n - n^+}(X \cap \R^{d-1}\times (r,\infty)) & \mbox{if } n > n^+.
    \end{cases}
  \end{equation*}
  Using the monotonicity from \lref{monotone}, we conclude
  \begin{equation*}
    \sum \id_{S^-_n} \leq s_X \leq \sum \id_{S^+_n}.
  \end{equation*}
  This implies
  \begin{equation*}
    n^- + s_{X \cap \R^{d-1} \times (r,\infty)}( (r+t) e_d) \leq s_X( (r+t) e_d) \leq n^+ + s_{X \cap \R^{d-1} \times (r,\infty)}( (r+t) e_d).
  \end{equation*}
  Since
  \begin{equation*}
    \E s_{X \cap \R^{d-1}\times (r,\infty)}( (r+t)e_d) = \E s_X(t e_d),
  \end{equation*}
  taking expectations yields the lemma.
\end{proof}
\nc

\subsection{Fluctuations}

We prove our main theorem about semiconvex peeling.

\begin{theorem}
  \label{t.semiconvex}
  There is a constant $\alpha > 0$ such that, if $X \sim Poisson(\id_H)$ and $r \geq t \geq 1$, then
  \begin{equation*}
    s_{X \cap Q_{Cr}} \leq s_X \leq s_{X \cup (H \setminus Q_{Cr})},
  \end{equation*}
  \begin{equation*}
    \P\left[\inf_{\partial^+ Q_r} s_{X \cap Q_{Cr}} \leq \alpha r - (\log r)^3 r^{1/2} t\right] \leq C \exp( - c t^{2/3}),
  \end{equation*}
  and
  \begin{equation*}
    \P\left[\sup_{\partial^+ Q_r} s_{X \cup (H \setminus Q_{Cr})} \geq \alpha r + (\log r)^3 r^{1/2} t\right] \leq C \exp( - c t^{2/3}).
  \end{equation*}
\end{theorem}

\begin{proof}
\red
Define $\alpha>0$ by 
\[\alpha = \liminf_{r\to \infty} \frac{1}{r}\E[s_X(re_d)].\]
By \lref{lower} and \lref{upper} we have $0 < \alpha < \infty$. We claim that 
\begin{equation}\label{e.limit}
|\E[s_X(r e_d)] - \alpha r| \leq C (\log r)^3 r^{1/2}.
\end{equation}
The proof of \eref{limit} is split into two parts.

1. We first show that 
\begin{equation}\label{e.alphalim}
\alpha = \lim_{r\to \infty} \frac{1}{r}\E[s_X(re_d)].
\end{equation}
To see this, define $g(r) = \E[S_X(re_d)] + r^{3/4}$. Applying \lref{additivity} with $t\geq r \geq C$ we deduce
\begin{align*}
g(t+r) &= g(t) + g(r) + \E[S_X( (t+r)e_d) - S_X(re_d) - S_X(te_d)]\\
&\hspace{2in} +(t+r)^{3/4}- r^{3/4} - t^{3/4}  \\
&\leq g(t) + g(r) + C(\log r)^3r^{1/2} +(t+r)^{3/4} - r^{3/4} - t^{3/4} \\
&\leq g(t) + g(r) + C(\log r)^3 r^{1/2} - r^{3/4}(1 - \tfrac34 (r/t)^{1/4})\\
&\leq g(t) + g(r) + C(\log r)^3 r^{1/2} - \frac14 r^{3/4}\leq g(t)+g(r)
\end{align*}
for $C$ sufficiently large. It follows that for any $r\geq C$ and $n\in \N$ we have $g(nr) \leq ng(r)$. Now, let $\eps>0$ and choose $r_0\geq C$ so that $g(r_0) \leq (\alpha + \eps)r_0$. Let $r>2r_0$ and write $r = nr_0 + k$ where $n\in \N$ and $r_0 \leq k \leq 2r_0$. We have
\[g(r) = g(nr_0 + k) \leq g(nr_0) + g(k) \leq ng(r_0) + g(k)\leq (\alpha+\eps)nr_0 + g(k).\]
It follows that $\limsup_{r\to \infty} g(r)/r \leq \alpha + \eps$, which completes the proof of \eref{alphalim}.

2. We now prove \eref{limit}. We note that \lref{additivity} with $t=r\geq C$ yields
\[\E[|s_X( 2re_d) - 2s_X(r e_d)|] \leq C (\log r)^3 r^{1/2}.\]
Thus, for any $k\geq 1$ and $r\geq C$ we have
\begin{align*}
\E[|s_X(2^kre_d) - 2^ks_X(re_d)|]&=\E\left[ \left|\sum_{i=1}^k 2^{k-i}s_X(2^ire_d) - 2^{k-i+1}s_X(2^{i-1}re_d)  \right|\right]\\
&\leq \sum_{i=1}^k 2^{k-i}\E[|s_X(2^ire_d) - 2s_X(2^{i-1}re_d)|]\\
&\leq C \sum_{i=1}^k 2^{k-i}(\log(2^{i-1}r))^3 (2^{i-1}r)^{1/2}\\
&\leq C(\log r)^3 2^k r^{1/2}.
\end{align*}
Therefore,
\begin{align*}
|\E[s_X(r e_d)] - \alpha r| &= |\E[s_X(re_d) - 2^{-k}s_X(2^kre_d) + 2^{-k}s_X(2^kre_d) -\alpha r] |\\
&\leq C(\log r)^3 r^{1/2}+  |\E[2^{-k}s_X(2^kre_d)] -\alpha r |.
\end{align*}
Sending $k\to \infty$ and invoking \eref{alphalim} completes the proof of \eref{limit}.

 \nc 
  We now apply \lref{periodize} and \lref{fluct} to find that
  \begin{equation*}
    \P\left[s_{X \cap Q_{Cr}}(r e_d) \leq \alpha r- (\log r)^3 r^{1/2} t\right] \leq C \exp( - c t^{2/3})
  \end{equation*}
  and
  \begin{equation*}
    \P\left[s_{X \cup (H \setminus Q_{Cr})}(r e_d) \geq \alpha r + (\log r)^3 r^{1/2} t\right] \leq C \exp( - c t^{2/3}).
  \end{equation*}
  A union bound over polynomially many points in $\partial^+ Q_r$ yields the theorem.
\end{proof}

\section{Viscosity Solutions}
\label{s.viscosity}

\subsection{Existence and uniqueness}

We now discuss the basic theory of the limiting equation \eref{bvp}.  We assume the reader is familiar with \mycite{Crandall-Ishii-Lions}.  We use viscosity solutions to interpret the non-linear partial differential equation
\begin{equation}
  \label{e.pde}
  \langle Dh, \cof(-D^2 h) Dh \rangle = f^2 \quad \mbox{in } U,
\end{equation}
where $U \subseteq \R^d$ is open and bounded and $f \in C(U)$ is non-negative.

While the left-hand side of \eref{pde} is not elliptic for general functions, it is elliptic on the set of quasi-concave functions.  That is, the functions $u$ whose super level set $\{ u > k \}$ is convex for all $k \in \R$.  This is a natural class of functions for our study.  In order to use standard viscosity machinery, we modify the operator outside the domain of ellipticity.

\begin{lemma}
  \label{l.Fdef}
  The function $F : \R^d \times \R^{d \times d}_{sym} \to \R$ defined by
  \begin{equation*}
    F(p,A) = \begin{cases}
      \langle p, \cof(-A) p \rangle & \mbox{if } \langle q, p \rangle = 0 \Rightarrow \langle q, A q \rangle \leq 0  \\
      0 & \mbox{otherwise}
    \end{cases}
  \end{equation*}  
  is continuous.  If $p \in \R^d$, $A, B \in \R^{d \times d}_{sym}$, and $A \leq B$, then $F(p,A) \geq F(p,B)$.  If $p \in \R^d$, $A \in \R^{d \times d}_{sym}$ and $B \in \R^{d \times d}$, then $F(B^t p, B^t A B) = \det(B)^2 F(p, A)$.
\end{lemma}

\begin{proof}
  When $p \neq 0$, the expression $\langle p, \cof(-A) p \rangle$ computes the determinant of $-A$ restricted to the subspace $p^\perp = \{ q \in \R^d : \langle q, p \rangle = 0 \}$.  Observe that this determinant is zero on the boundary of the set where the constraint $\langle q, p \rangle \Rightarrow \langle q, A q \rangle \leq 0$ holds.  It follows that $F$ is continuous.  Since $A$ is non-positive on $p^\perp$ when the constraint holds, it follows that $F$ is non-increasing in $A$.  For the last property, we use the continuity of $F$ to assume that $A, B$ are invertible.  We compute
  \begin{equation*}
    \begin{aligned}
      \langle B^t p, \cof(-B^t A B) B^t p \rangle
      & = \langle B^t p, \det(-B^t A B) (-B^t A B)^{-1} B^t p \rangle \\
      & = \det(B)^2 \langle p, \det(-A) (-A)^{-1} p \rangle \\
      & = \det(B)^2 \langle p, \cof(-A) p \rangle.
    \end{aligned}
  \end{equation*}
  Similarly, we see that the condition $\langle q, p \rangle = 0 \Rightarrow \langle q, A q \rangle \leq 0$ is equivalent to the condition $\langle q, B^t p \rangle = 0 \Rightarrow \langle q, B^t A B q \rangle \leq 0$.
\end{proof}

We obtain comparison when $f$ is positive by Ishii's lemma.

\begin{theorem}
  \label{t.comparison}
  If $U \subseteq \R^d$ is open and bounded, $f \in C(\overline U)$ satisfies $f > 0$ on $\overline U$, and $u \in USC(\overline U)$ and $v \in LSC(\overline U)$ are, respectively, a viscosity subsolution and supersolution of
  \begin{equation*}
    F(Dh, D^2 h) = f^2 \quad \mbox{in } U,
  \end{equation*}
  then $\max_{\overline U} (u - v) = \max_{\partial U} (u - v)$.
\end{theorem}

\begin{proof}
  Let us suppose for contradiction that the conclusion fails.  In this case, we may choose $\tau > 1$ and $\ep > 0$ such that $\max_{\overline U} (u - \tau v) = \ep + \max_{\partial U} (u - \tau v)$.  Note that $\tau v$ is a viscosity subsolution of
  \begin{equation*}
    F(Dh, D^2 h) = \tau^{d+1} f^2 \quad \mbox{in } U.
  \end{equation*}
  Since $f > 0$ on the closed set $\overline U$, there is a $\delta > 0$ such that $\tau^{d+1} f^2 \geq \delta + f^2$ on $\overline U$.  We now need only prove strict comparison; see \mycite{Crandall-Ishii-Lions}.
\end{proof}

\begin{remark}
  The above comparison result holds without imposing any quasi-concavity hypothesis on $u$ or $v$.  This works because the positivity of $f$ forces the supersolution $u$ to be quasi-concave; see \mycite{Barron-Goebel-Jensen}.  We expect that comparison theorem holds for $f \geq 0$ when the supersolution $u$ is quasi-concave.  This would require a deeper adaptation of the viscosity tools.
\end{remark}

To prove existence of solutions to our boundary value problem \eref{bvp}, we need barrier functions to show that the boundary values are attained.  Since the zero function is a subsolution, we need only obtain upper barriers.

\begin{lemma}
  \label{l.barrier}
  The function
  \begin{equation*}
    \psi(x) = 2 x_d^{\frac{2}{d+1}}(1 - \tfrac12 |x^d|^2)^{\frac{d-1}{d+1}},
  \end{equation*}
  where
  \begin{equation*}
    x^d = (x_1, ..., x_{d-1}),
  \end{equation*}
  satisfies
  \begin{equation}
    \label{e.barrier}
    \begin{cases}
      F(D \psi, D^2 \psi) \geq 1 & \mbox{in } B_1 \cap \{ x_d > 0 \} \\
      \psi \geq 0 & \mbox{on } \overline B_1 \cap \{ x_d \geq 0 \} \\
      \psi = 0 & \mbox{on } B_1 \cap \{ x_d = 0 \}.
    \end{cases}
  \end{equation}
\end{lemma}

\begin{proof}
  For $t > 0$, consider the function
  \begin{equation*}
    \psi_t(x) = t^{1-d} x_d + t^2 (1 - \tfrac12 |x^d|^2).
  \end{equation*}
  Observe that $\psi_t$ satisfies \eref{barrier} classically.  Compute
  \begin{equation*}
    \psi(x) = \inf_{t > 0} \psi_t(x) = \psi_{t(x)}(x),
  \end{equation*}
  where
  \begin{equation*}
    t(x) = x_d^{1/(d+1)}(1 - \tfrac12 |x^d|^2)^{-1/(d+1)}
  \end{equation*}
  Since $\psi$ is continuous, the ellipticity of $F$ implies that $\psi$ satisfies \eref{barrier} in the sense of viscosity.  Since $\psi$ is smooth in $B_1 \cap \{ x_d > 0 \}$, it also satisfies \eref{barrier} classically.
\end{proof}

We obtain existence by a standard application of Perron's method.

\begin{theorem}
  \label{t.existence}
  Suppose $U \subseteq \R^d$ is bounded open and convex and $f \in C(\bar U)$ satisfies $f > 0$.  There is a unique $u \in C(\bar U)$ that satisfies
  \begin{equation}
    \label{e.bvpF}
    \begin{cases}
      F(D u, D^2 u) = f^2 & \mbox{in } U \\
      u = 0 & \mbox{on } \partial U
    \end{cases}
  \end{equation}
  in the sense of viscosity.
\end{theorem}

\begin{proof}
  Rescaling, we may assume that $U \subseteq B_{1/2}$ and $f \leq 1$.  For every $p \in \partial U$ with inward normal $n_p \in \R^d$, choose an orthogonal matrix $O_p \in \R^{d \times d}$ such that $O_p n_p = e_d$.  Using \lref{Fdef} and \lref{barrier}, we see that the functions $\psi_p(p + x) = \psi(O_p x)$ are supersolutions of \eref{bvpF} that satisfy $\psi_p(p) = 0$.  The zero function is a subsolution of \eref{bvpF} that achieves the boundary conditions.  Since we have a comparison principle from \tref{comparison}, the supremum of all subsolutions is equal to the infimum of all supersolutions, and this object is the unique solution of \eref{bvp}.
\end{proof}

Another application of our barrier is H\"older regularity.

\begin{corollary}
  \label{c.holder}
  The unique solution $u \in C(\bar U)$ from \tref{existence} satisfies the H\"older estimate $\| u \|_{C^{2/(d+1)}(U)} \leq \lalpha$, where $\lalpha$ depends only on $\diam U$ and $\max f$.
\end{corollary}

\begin{proof}
  Rescaling, we may assume that $U \subseteq B_{1/2}$ and $f \leq 1$.  By \mycite{Barron-Goebel-Jensen}, $u$ is quasi-concave.  Suppose $x, y \in U$ and $u(x) < u(y)$.  Let $V = \{ u > u(x) \}$, which is convex and open.  Choose $z \in \partial V$ such that $|y - z| = \dist(y, \partial V)$.  Choose an orthogonal matrix $O \in \R^{d \times d}$ such that $O (y - z) = |y - z| e_d$.  Let $\tilde \psi(z + w) = u(x) + \psi(O w)$, where $\psi$ is from \lref{barrier}.  Note that $F(D \tilde \psi, D^2 \tilde \psi) \geq 1$ in $V$ and $\tilde \psi \geq k \geq u$ on $\partial V$.  By \tref{comparison}, we obtain $u(y) \leq \tilde \psi(y)$.  In particular, $u(y) \leq \tilde \psi(y) = u(x) + 2 |y - z|^{2/d+1} \leq u(x) + 2 |y - x|^{2/d+1}$.
\end{proof}

\subsection{Simple Test Functions}

We construct a family of simple test functions that form a complete family for the operator $F$.  Recall the function
\begin{equation*}
  \varphi(x) = x_d - \tfrac12 (x_1^2 + \cdots + x_{d-1}^2)
\end{equation*}
which satisfies $F(D \varphi, D^2 \varphi) = 1$.  We build our test functions by distorting $\varphi$.

\begin{definition}
  A simple upper test function is a function of the form
  \begin{equation*}
    \psi = \sigma \circ \varphi \circ a,
  \end{equation*}
  where $\sigma \in C^\infty(\R)$, $\sigma' \geq 0$, $\sigma'' \geq 0$, $a \in C^\infty(\R^d,\R^d)$, $D a$ constant, and $\det Da = 1$.
\end{definition}

\begin{definition}
  A simple lower test function is a function of the form
  \begin{equation*}
    \psi = \sigma \circ \varphi \circ a,
  \end{equation*}
  where $\sigma \in C^\infty(\R)$, $\sigma' \geq 0$, $\sigma'' \leq 0$, $a \in C^\infty(\R^d,\R^d)$, $D a$ constant, and $\det Da = 1$.
\end{definition}

The following formalizes what we mean by complete family.

\begin{lemma}
  \label{l.testfunctions}
  Suppose $u \in C^\infty(\R^d)$ and $F(Du(x),D^2 u(x)) > 0$.
  \begin{enumerate} 
  \item For every small $\ep > 0$, there is a $\delta > 0$ and a simple upper test function $\psi$ such that $\psi(x) = u(x)$, $\psi(y) > u(y)$ for $0 < |y-x| < \delta$, and $F(D \psi(x), D^2 \psi(x)) \leq (1 + \ep) F(Du(x),D^2 u(x))$.

  \item For every small $\ep > 0$, there is a $\delta > 0$ and a simple lower test function $\psi$ such that $\psi(x) = u(x)$, $\psi(y) < u(y)$ for $0 < |y-x| < \delta$, and $F(D \psi(x), D^2 \psi(x)) \geq (1 - \ep) F(Du(x),D^2 u(x))$.
  \end{enumerate}
\end{lemma}

\begin{proof}
  Part 1.  Observe that $Du(x) \neq 0$ and that $D^2 u(x)$ is negative definite on the half space orthogonal to $D u(x)$.  Using \lref{Fdef} and the definition of simple test function, we can make an affine change of variables so that
  \begin{equation*}
    x = 0,
    \quad
    D u(0) = |Du(0)| e_d,
    \quad \mbox{and} \quad
    D^2 u(0) = \left[\begin{matrix} - (1-\ep) |D u(0)| I_{d-1} & v \\ v^t & \gamma \end{matrix}\right],
  \end{equation*}
  where $v \in \R^{d-1}$ and $\gamma \in \R$.  For $\beta > 0$ to be determined $\psi = \sigma \circ \varphi$, where
  \begin{equation*}
    \sigma(s) = u(0) + \beta^{-1}|D u(0)|(e^{\beta s} - 1).
  \end{equation*}
  Note that $\sigma \in C^\infty(\R)$, $\sigma' \geq 0$, and $\sigma'' \geq 0$.  Moreover,
  \begin{equation*}
    \psi(0) = u(0),
    \quad
    D \psi(0) = D u(0),
    \quad \mbox{and} \quad
    D^2 \psi(0) = \left[ \begin{matrix}
        - |D u(0)| I_{d-1} & 0 \\ 0 & \beta |D u(0)|.
      \end{matrix}\right]
  \end{equation*}
  Making $\beta > 0$ large, we obtain $D^2 \psi(0) > D^2 u(0)$.  By second order expansion, we can choose $\delta > 0$ so that $\psi(y) > u(y)$ for $0 < |y - x| < \delta$.  Finally, compute $F(D \psi(0),D^2 \psi(0)) = (1-\ep)^{1-d} F(D u(0),D^2 u(0)).$

  Part 2.  Observe that $Du(x) \neq 0$ and that $D^2 u(x)$ is negative definite on the half space orthogonal to $D u(x)$.  Using \lref{Fdef} and the definition of simple test function, we can make an affine change of variables so that
  \begin{equation*}
    x = 0,
    \quad
    Du(0) = |Du(0)| e_d,
    \quad \mbox{and} \quad
    D^2 u(0) = \left[\begin{matrix} - (1+\ep) |D u(0)| I_{d-1} & v \\ v^t & \gamma \end{matrix}\right],
  \end{equation*}
  where $v \in \R^{d-1}$ and $\gamma \in \R$. For $\beta > 0$ to be determined $\psi = \sigma \circ \varphi$, where
  \begin{equation*}
    \sigma(s) = u(0) + \beta^{-1} |Du(0)| (1 - e^{-\beta s}).
  \end{equation*}
  Note that $\sigma \in C^\infty(\R)$, $\sigma' \geq 0$, and $\sigma'' \leq 0$.  Moreover,
  \begin{equation*}
    \psi(0) = u(0),
    \quad
    D \psi(0) = D u(0),
    \quad \mbox{and} \quad
    D^2 \psi(0) = \left[ \begin{matrix}
        - |D u(0)| I_{d-1} & 0 \\ 0 & - \beta |D u(0)|.
      \end{matrix}\right]
  \end{equation*}
  Making $\beta > 0$ large, we obtain $D^2 \psi(0) < D^2 u(0)$.  By second order expansion, we can choose $\delta > 0$ so that $\psi(y) < u(y)$ for $0 < |y - x| < \delta$.  Finally, compute $F(D \psi(0),D^2 \psi(0)) = (1+\ep)^{1-d} F(D u(0),D^2 u(0)).$
\end{proof}
  
\subsection{Piece-wise approximation}

For the purposes of proving the scaling limit of convex peeling, it is useful to recall the viscosity analogue of Galerkin approximation.  When there is a comparison principle, Perron's method implies more than just the existence of a solution.  In fact, it implies that the solution is the uniform limit of piece-wise smooth subsolutions and supersolutions.  We obtain a slightly stronger version where the pieces are all simple upper or lower test functions.

\begin{definition}
  A piece-wise supersolution of $F(Dh,D^2h) = f^2$ in $U$ is a function $u \in C(U)$ for which there is a finite list of simple upper test functions $\psi_k$ and balls $B_{r_k}(x_k)$ such that
  \begin{enumerate}
  \item $\psi_k \geq u$ in $B_{r_k}(x_k) \cap \bar U$,
  \item $F(D \psi_k, D^2 \psi_k) > \sup_{B_{r_k}(x_x) \cap U} f^2$ in $B_{r_k}(x_k)$,
  \item for every $x \in \bar U$, there is a $k$ such that $x \in B_{r_k/3}(x_k)$ and $u(x) = \psi_k(x)$.
  \end{enumerate}
\end{definition}

\begin{definition}
  A piece-wise subsolution of $F(Dh,D^2h) = f^2$ in $U$ is a function $u \in C(\bar U)$ for which there is a finite list of simple lower test functions $\psi_k$ and balls $B_{r_k}(x_k)$ such that
  \begin{enumerate}
  \item $\psi_k \leq u$ in $B_{r_k}(x_k) \cap \bar U$,
  \item $F(D \psi_k, D^2 \psi_k) < \inf_{B_{r_k}(x_x) \cap U} f^2$ in $B_{r_k}(x_k)$,
  \item for every $x \in \bar U$, there is a $k$ such that $x \in B_{r_k/3}(x_k)$ and $u(x) = \psi_k(x)$.
  \end{enumerate}
\end{definition}

Observe that a piece-wise supersolution is a viscosity supersolution and that the set of piece-wise supersolutions is closed under pairwise minimum.  The analogous facts are true for piece-wise subsolutions.

Before proving a general approximation result, observe that \lref{testfunctions} only provides simple approximations when $F(D u(x),D^2 u(x)) > 0$.  Finding a piece-wise approximation when $Du(x) = 0$ requires an ad hoc argument.

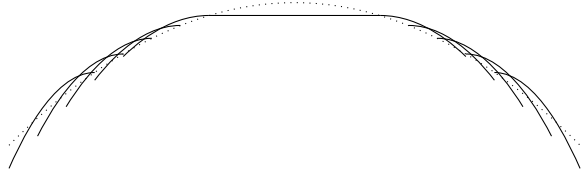
\begin{figure}[h]
  \begin{tikzpicture}[scale=0.3\textwidth,x=1,y=1]
  \draw[dotted] (-1,-1/2) parabola bend (0,0) (1,-1/2);
  \foreach \h in {.3} {
    \foreach \x in {.3,.4,.5,.6,.7} {
      \draw (\x+\h,-0.5*\x*\x-1.6*\x*\h) parabola bend (\x,-0.5*\x*\x) (\x,-0.5*\x*\x);
    }
    \foreach \x in {-.7,-.6,-.5,-.4,-.3} {
      \draw (\x-\h,-0.5*\x*\x+1.6*\x*\h) parabola bend (\x,-0.5*\x*\x) (\x,-0.5*\x*\x);
    }
  }
  \draw (-.3,-0.5*.3*.3) -- (.3,-0.5*.3*.3);
\end{tikzpicture}
  \caption{The piece-wise approximation of a downward parabola by simple lower test functions.}
  \label{f.zerofunction}
\end{figure}

\begin{lemma}
  \label{l.zerofunction}
  Suppose $\lalpha > 0$ and consider the function
  \begin{equation*}
    u(x) = - \tfrac{d+1}{2d} \lalpha |x|^{\tfrac{2d}{d+1}},
  \end{equation*}
  which satisfies $F(Du(x),D^2 u(x)) = \lalpha^{d+1}$ in $\R^d \setminus \{ 0 \}$.  For every $R, \ep > 0$, there is a piece-wise subsolution $v \in C(B_R)$ of $F(Dv,D^2 v) = \lalpha^{d+1} + \ep$ in $B_R$ such that $|v - u| < \ep$ in $B_R$.
\end{lemma}

\begin{proof}
  For every $x \in \R^d \setminus \{ 0 \}$, we can use \lref{testfunctions} to select a simple \red upper \nc test function $\psi_x$ and a radius $r_x > 0$ such that $\psi_x \leq u + \ep$, $0 \in \{ \psi_x > u \} \subseteq B_{r_x}(x)$, and $F(D\psi_x,D^2\psi_x) < \lalpha^{d+1} + \ep$ in $B_{r_x}(x)$.  Indeed, we take that $\psi$ the lemma produces and add a small positive constant.  Select any $r > 0$ such that $u(x) > -\tfrac12 \ep$ in $B_r$.  By compactness, we may select finitely many $x_k$ such that
  \begin{equation*}
    \bar B_R \setminus B_r \subseteq \cup_k \{ \psi_{x_k} > u \}.
  \end{equation*}
  Consider the function
  \begin{equation*}
    v(x) = \max \{ u(x), \max \{ \psi_{x_k}(x) : x \in B_{r_k}(x) \} \}.
  \end{equation*}
  Now, $v$ is a viscosity subsolution of $F(Dv,D^2 v) = \lalpha^{d+1} + \ep$.  Moreover, $v$ is a piece-wise subsolution in $B_R \setminus \bar B_r$. To fix the piece in $B_r$, we select a $\sigma \in C^\infty(\R)$ such that $\sigma' \geq 0 \geq \sigma''$, $\sigma(s) = s$ if $s \leq -\ep$ and $\sigma(s) = -\tfrac12 \ep$ if $s \geq -\tfrac12\ep$.  Then $w = \sigma \circ v$ is a piece-wise subsolution of $F(D w,D^2 w) = \lalpha^{d+1} + \ep$ in $B_R$.  A schematic of $w$ appears in \fref{zerofunction}.
\end{proof}

Now that we can approximate test functions whose gradient vanishes, we prove our general approximation result.

\begin{theorem}
  \label{t.piecewise}
  Let $u, f \in C(\bar U)$ be as in \tref{existence}.  For any $\ep > 0$, there is a piece-wise supersolution $\overline u \in C(\bar U)$ and a piecewise subsolution $\underline u \in C(\bar U)$ such that $u - \ep \leq \underline u \leq u \leq \overline u \leq u + \ep$ in $\bar U$.
\end{theorem}

\begin{proof}
  By the comparison result from \tref{comparison}, it is enough to show that the infimum of all piecewise supersolutions is a subsolution and that the supremum of all piecewise subsolutions is a supersolution.  This would be a folklore theorem were it not for the fact that we demand the pieces have a special form.

  We first consider the subsolution case.  Let
  \begin{equation*}
    \underline u = \sup \{ v \in C(\bar U) \mbox{ a piecewise subsolution of \eref{bvpF}} \}
  \end{equation*}
  and suppose for contradiction that $u$ is not a supersolution.  Since $0$ is a piecewise subsolution, we see that $\underline u \geq 0$.  Using \lref{barrier} and \tref{comparison}, we see that $\sup_{\overline U} \underline u < \infty$ and $\underline u \leq 0$ on $\partial U$.  Since $\underline u$ is a bounded supremum of continuous functions, it is lower semicontinuous.  Thus, the supersolution condition must fail in the interior and we may select $B_r(x) \subseteq U$ and smooth $w \in C^\infty(B_r(x))$ such that $w(x) = \underline u(x)$, $w < \underline u$ in $B_r(x) \setminus \{ y \}$, and $F(D w, D^2 w) < \inf_{B_r(x)} f^2$ in $B_r(x)$.

  Since $F$ is continuous, we may replace $w$ by $y \mapsto w(y) - \beta |y - x|^2$ for some $\beta > 0$ so that $D^2 w(x)$ is negative definite on the subspace $\{ q \in \R^d : q \cdot Dw(x) = 0 \}$.  In this case, we see that either $Dw(x) = 0$ or $F(Dw(x),D^2w(x)) > 0$.  Applying either \lref{testfunctions} or \lref{zerofunction}, we can make $r > 0$ smaller and replace $w$ with a piecewise subsolution of $F(Dh, D^2 h) = f^2$ in $B_r(x)$.

  Since $\underline u$ is merely lower semicontinuous at this stage of the proof, we do not know how to choose piecewise subsolutions such that $v_k \in C(\bar U)$ such that $v_k \to \underline u$ uniformly.  However, since $w$ is continuous, for any compact $K \subseteq \{ \underline u > w \}$, we can choose a piecewise subsolution $v \in C(\bar U)$ such that $v > w$ on $K$.  Indeed, we find a piecewise subsolution above $\varphi$ in a neighborhood of every point in $K$, choose a finite cover, and then compute the maximum of the finite set of piecewise subsolutions.

  We choose a piecewise subsolution $v \in C(\bar U)$ and $\delta > 0$ such that $v \leq u$ and $v > w + 2 \delta$ on $B_r(x) \setminus B_{r/2}(x)$.  We then define
  \begin{equation*}
    v'(y) = \begin{cases}
      \max \{ v(y), (w + \delta)(y) \} & \mbox{if } y \in B_r(x) \\
      v(y) & \mbox{otherwise},
    \end{cases}
  \end{equation*}
  which is a piecewise subsolution of the global problem satisfying $v(x) > \underline u(x)$, contradicting the definition of $\underline u$.

  The supersolution case is symmetric and easier, since $F(Dw(x),D^2w(x)) > f(x)^2$ implies that $Dw(x) \neq 0$.
\end{proof}

\begin{remark}
  Our naive use of compactness in the above proof destroys any hope of quantifying the number of pieces in the approximation.  However, it is clear from the definitions that the number of pieces depends on the regularity of the solution.
\end{remark}

\section{Convex Peeling}
\label{s.convex}

\subsection{Comparison lemmas}

We now explain the relation between convex and semiconvex peeling.  Recall the parabolic region $P$ and half space $H$ defined in \sref{semiconvex}.  Consider the bijection $\pi : P \to H$ given by
\begin{equation*}
  \pi(x) = (x_1, ..., x_{d-1}, x_d - \tfrac12 (x_1^2 + \cdots x_{d-1}^2)).
\end{equation*}
Since $\det D \pi = 1$, if $X \sim Poisson(\id_P)$, then $\pi(X) \sim Poisson(\id_H)$.  Moreover, the sets $(y - P) \cap H$ for $y \in H$ are exactly the sets $\pi(P \cap \tilde H)$ where $\tilde H \subseteq \R^d$ is a half space such that $P \cap \tilde H$ is bounded.  From this it follows that, if $X \subseteq P$ contains a sequence $\{ x_n \} \subseteq X$ with $e_d \cdot x_n \to \infty$, then
\begin{equation*}
  \pi(K_n(X)) = S_n(\pi(X)) \quad \mbox{and} \quad s_{\pi(X)} \circ \pi = h_X.
\end{equation*}
In particular, if $X \sim Poisson(\id_P)$, then the above holds almost surely.

Using the monotonicity from \lref{monotone} and its immediate analogue for convex peeling, we prove a local connection between convex and semiconvex peeling.  Both of the following lemmas make use of the geometry illustrated in \fref{pitransform}.

\begin{figure}[h]
  \begin{tikzpicture}[scale=.04\textwidth,x=1,y=1]
  \draw[dotted] (-4,4) -- (4,4);
  \draw[dotted] (-3.46,2) -- (3.46,2);
  \draw[dotted] (-2.83,0) -- (2.83,0);
  \draw (-4,8) parabola bend (0,0) (4,8);
  \draw (-3.46,8) parabola bend (0,2) (3.46,8);
  \draw (-2,2) -- (-2,4);
  \draw (2,2) -- (2,4);
  \draw (0,1.5) node {$\pi^{-1}(Q_2)$};
  
  \begin{scope}[xshift=10]  
    \draw[dotted] (-4,-4) parabola bend (0,4) (4,-4);
    \draw[dotted] (-3.46,-4) parabola bend (0,2) (3.46,-4);
    \draw[dotted] (-2.83,-4) parabola bend (0,0) (2.83,-4);
    \draw (-4,0) -- (4,0);
    \draw (-3.46,2) -- (3.46,2);
    \draw (-2,0) -- (-2,2);
    \draw (2,0) -- (2,2);
    \draw (0,1) node {$Q_2$};
  \end{scope}
  
  \draw[->] (4.5,2) -- (5.5,2);
  \draw (5,2.5) node {$\pi$};
\end{tikzpicture}
  \caption{The local behavior of the transformation $\pi$.}
  \label{f.pitransform}
\end{figure}
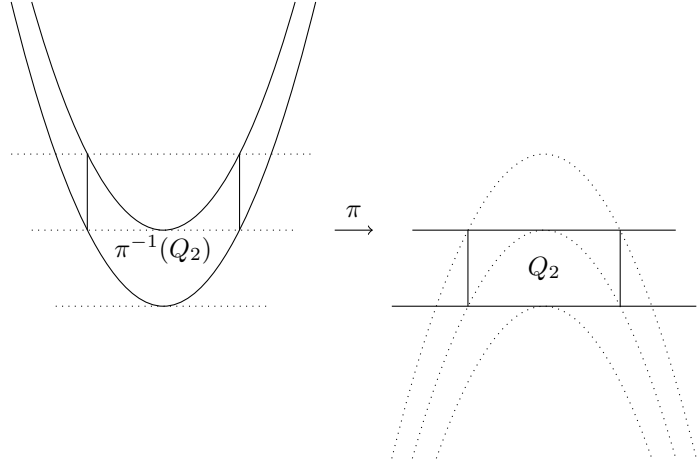

\begin{lemma}
  \label{l.compareup}
  If $X \subseteq \R^d$ and
  \begin{equation*}
    K_1(X) \subseteq P \cup (2 e_d + H),
  \end{equation*}
  then
  \begin{equation*}
    h_X \leq s_{(\pi(X) \cap Q_2) \cup (H \setminus Q_2)} \circ \pi \quad \mbox{in } \pi^{-1}(Q_2)
  \end{equation*}
\end{lemma}

\begin{proof}
  Observe that, if $H \setminus Q_2 \subseteq Y \subseteq H$, then
  \begin{equation*}
    s_{Y \cup (\R^d \setminus (2 e_d - P))} = s_Y \quad \mbox{in } H.
  \end{equation*}
  Now consider $Y = (\pi(X) \cap Q_2) \cup (H \setminus Q_2)$.  Using the hypothesis $K_1(X) \subseteq P$, we obtain
  \begin{equation*}
    h_{X \cup (2 e_d + H)} = s_{Y \cup (\R^d \setminus (2 e_d - P))} \circ \pi.
  \end{equation*}
  By monotonicity
  \begin{equation*}
    h_X \leq h_{X \cup (2e_d + H)}.
  \end{equation*}
  Conclude by combining the above three observations.
\end{proof}

\begin{lemma}
  \label{l.comparedown}
  If $n\geq 1$, $X \subseteq \R^d$ and
  \begin{equation*}
    K_n(X) \supseteq (2 e_d + P) \setminus (4 e_d + H),
  \end{equation*}
  then
  \begin{equation*}
    h_X \geq s_{\pi(X) \cap Q_2} \circ \pi \quad \mbox{in } \pi^{-1}(Q_2) \setminus K_n(X).
  \end{equation*}
\end{lemma}

\begin{proof}
  Consider the intersection of a parabolic region and a cylinder
  \begin{equation*}
    Q = \{ x \in 2 e_d + P : x_1^2 + \cdots + x_{d-1}^2 < 4 \}.
  \end{equation*}
  The hypothesis $K_n(X) \supseteq (2 e_d + P) \setminus (4 e_d + H)$ implies that
  \begin{equation*}
    h_X = h_{X \cup Q} \quad \mbox{in } \pi^{-1}(Q_2) \setminus K_n(X).
  \end{equation*}
  By the discussion above,
  \begin{equation*}
    h_{X \cup Q} = s_{\pi(X \cup Q)} \circ \pi.
  \end{equation*}
  By monotonicity,
  \begin{equation*}
    s_{\pi(X \cup Q)} \circ \pi \geq s_{\pi(X) \cap Q_2} \circ \pi.
  \end{equation*}
  Conclude by combining the above three observations.
\end{proof}

\subsection{Local height functions}

Combining the above comparison lemmas and the fluctuation bounds from \tref{semiconvex}, we constrain the local behavior of the convex height function of a Poisson cloud.  Recall that if $X \sim Poisson(\id_P)$, then $\pi(X) \sim Poisson(H)$ and $s_{\pi(X)} \circ \pi = h_X$ holds almost surely.  In particular, \tref{semiconvex} suggests that $h_X \approx \max \{ 0, \alpha \varphi \}$, where
\begin{equation*}
  \varphi(x) = x_d - \tfrac12 (x_1^2 + \cdots + x_{d-1}^2).
\end{equation*}
Using the comparison lemmas, we use this idea to show that $\varphi$ forms a local barrier for convex height functions.

\begin{definition}
  A local height function for a set $X \subseteq B_1$ is a function $h : B_1 \to \N$ such that $h = h_Y | B_1$ for some finite set $Y \subseteq \R^d$ that satisfies $Y \cap B_1 = X$.  Let $\mathcal{H}(X)$ denote the set of local height functions of $X \subseteq B_1$.
\end{definition}

We consider perturbations of $\varphi$ of the form $\tilde \varphi = \sigma \circ \varphi$ where $\sigma \in C^1(\R)$ satisfies either $\sigma' > 1$ or $0 < \sigma' < 1$.  Note that $\tilde \varphi$ has the same level sets as $\varphi$, but they evolve at different rates. We show the two types of perturbations form upper and lower barriers, respectively.

\begin{lemma}
  \label{l.upperbarrier}
  If $\sigma \in C^\infty(\R)$ satisfies $\sigma' > 1 + \lambda>1$ and $\sigma'' \geq 0$, $m > 2$, $X \sim Poisson(m \id_{B_1})$, and $\psi = \alpha m^{\frac{2}{d+1}} \sigma \circ \varphi$, then
  \begin{multline}
    \label{e.stricttouchabove}
    \P[\sup_{B_1} (h - \psi) = \sup_{B_{1/3}} (h - \psi) \mbox{ for some } h \in \mathcal{H}(X)] \\ \leq C \exp( - c \lambda^{2/3} (\log m)^{-2} m^{1/3(d+1)}).
  \end{multline}
\end{lemma}

\begin{figure}[h]
  \begin{tikzpicture}[scale=0.25*\textwidth,x=1,y=1]      
  
  \draw[anchor=west] (.85,.6) node {$\{ h \geq h(z) \}$};
  \draw[anchor=west] (.9,.4) node {$\{ \varphi \geq \varphi(w + m^{-2/(d+1)} r e_d) \}$};
  \draw[anchor=west] (.95,.2) node {$\{ h \geq h(z) - n \}$};
  \draw[anchor=west] (1,0) node {$\{ \varphi \geq \varphi(w) \}$};
  
  \draw[ultra thin,lightgray] (0,0) circle (1);
  \draw[ultra thin,lightgray] (0,0) circle (2/3);
  \draw[ultra thin,lightgray] (0,0) circle (1/3);

  \clip (0,0) circle (1);
    
  \draw (-1,1/2) parabola bend (0,0) (1,1/2);
  \draw (-1,0) parabola bend (0,-1/2) (1,0);

  %\draw (0,-1/2) parabola bend (0,-1/2) (1/2,1/8-1/2) -- (1/2,1/8) parabola bend (0,0) (0,0) -- (0,-1/2);
  
  \draw (1/4,-1/8) node {\textbullet};
  \draw (1/4,-1/8-.1) node {$z$};
  
  \draw (1/4,-15/32) node {\textbullet};
  \draw (1/4,-15/32-.1) node {$w$};

  \draw[dashed] (-1,2) -- (-1/2,1/2) -- (-1/4,1/8) -- (0,-1/8) -- (1/4,-1/8) -- (1/2,0) -- (1,2/3);
      
  \draw[dashed] (-1,2) -- (-1/2,1/4) -- (-1/4,-1/3) -- (0,-2/5) -- (1/4,-1/3) -- (1/2,-1/4) -- (1,1/6);

\end{tikzpicture}
  \caption{A diagram of the level sets in the proof of \lref{upperbarrier}.}
  \label{f.upperbarrier}
\end{figure}

\begin{proof}
  We use \lref{compareup} to show that the event in \eref{stricttouchabove} is contained in polynomially many events that are controlled by \tref{semiconvex}.  We may assume that $m \geq C$ is large in what follows.

  Define, for $z \in B_{2/3}$, the map
  \begin{equation*}
    \tau_z(x^d, x_d) = \left( m^{\frac{1}{d+1}} (x^d - z^d), m^{\frac{2}{d+1}} (x_d - z_d - z^d \cdot (x^d - z^d)) \right),
  \end{equation*}
  where $x^d = (x_1, ..., x_{d-1})$.  Observe that $\det D \tau_z = m$, $\tau_z(0) = 0$, and $\tau_z(\{ \varphi > \varphi(z) \}) = P$.  Moreover, for any $r \leq c m^{1/(d+1)}$, observe that $(\pi \circ \tau_z)^{-1}(Q_r) \subseteq B_1$.

  Let $\ep > 0$ be a universal constant determined later.  If the event in \eref{stricttouchabove} occurs, then we can choose $z, w \in B_{2/3}$, $r > 0$, and $n \in \N$ such that
  \begin{equation*}
    r = \ep m^{1/(d+1)},
  \end{equation*}
  \begin{equation*}
    z \in (\pi \circ \tau_w)^{-1}(Q_r),
  \end{equation*}
  \begin{equation*}
    \{ \varphi \geq \varphi(w) \} \supseteq B_{2/3} \cap \{ h \geq h(z) - n \},
  \end{equation*}
  and
  \begin{equation*}
    n \geq \alpha (1 + \tfrac12 \lambda) r.
  \end{equation*}
  This is depicted in \fref{upperbarrier}.  Moreover, we may select $w$ from a predetermined list of $C (m \lambda^{-1} \ep^{-1})^C$ many points in $B_{2/3}$.

  We now reduce to a large deviation event parameterized by $w$.  Applying \lref{compareup} to the point set $X \cap \{ h \geq h(z) - n \}$ and making $\ep > 0$ sufficiently small, we obtain
  \begin{equation*}
    \sup_{\partial^+ Q_r} s_{(\pi \circ \tau_w)(X) \cap Q_{Cr}} \geq \alpha (1 + \tfrac14 \lambda r).
  \end{equation*}
  Let $\tilde E_w$ denote the above event.  Since, when $m > 0$ is large and $\ep > 0$ is small, $(\pi \circ \tau_w)(X) \cap Q_{Cr} \sim Poisson(\id_{Q_{Cr}})$, \tref{semiconvex} implies
  \begin{equation*}
    \P[\tilde E_w] \leq C \exp( - \lambda^{2/3} c (\log r)^{-2} r^{1/3}).
  \end{equation*}
  Recalling that $r = \ep m^{1/(d+1)}$ and summing over the polynomially many $w$ yields the lemma.
\end{proof}

The next lemma would be symmetric to the previous lemma, were it not for the fact that we allow $\sigma' = 0$.  Flat spots are necessary for lower test functions, as demonstrated in the previous section.  Handling the flat spots requires an additional appeal to the Poisson law, which is used to control the set of points where the infimum in \eref{stricttouchbelow} is achieved to lie close to the non-flat part of the test functions.  With the geometry under control, we are able to deform from the case $0 \leq \sigma' < 1 - \lambda < 1$ to the case $0 < \sigma' < 1 - \tfrac12 \lambda < 1$.

\begin{lemma}
  \label{l.lowerbarrier}
  If $\sigma \in C^\infty(\R)$ satisfies $0 \leq \sigma' < 1 - \lambda < 1$ and $\sigma'' \leq 0$, $m > 2$, $X \sim Poisson(m \id_{B_1})$, and $\psi = \alpha m^{\frac{2}{d+1}} \sigma \circ \varphi$, then
  \begin{multline}
    \label{e.stricttouchbelow}
    \P[\inf_{B_1} (h - \psi) = \inf_{B_{1/3}} (h - \psi) \mbox{ for some } h \in \mathcal{H}(X)] \\ \leq C \exp( - c \lambda^{2/3} (\log m)^{-2} m^{1/3(d+1)}).
  \end{multline}
\end{lemma}

\begin{figure}[h]
  \begin{tikzpicture}[scale=0.25*\textwidth,x=1,y=1]      
  
  \draw[anchor=west] (.9,.4) node {$\{ \varphi \geq \varphi(w + m^{-2/(d+1)} r e_d) \}$};
  \draw[anchor=west] (.95,.2) node {$\{ h \geq h(z) \}$};
  \draw[anchor=west] (1,0) node {$\{ \varphi \geq \varphi(w) \}$};
  \draw[anchor=west] (1,-.2) node {$\{ h \geq h(z) - n \}$};
  
  \draw[ultra thin,lightgray] (0,0) circle (1);
  \draw[ultra thin,lightgray] (0,0) circle (2/3);
  \draw[ultra thin,lightgray] (0,0) circle (1/3);

  \clip (0,0) circle (1);
    
  \draw (-1,1/2) parabola bend (0,0) (1,1/2);
  \draw (-1,0) parabola bend (0,-1/2) (1,0);

  \draw (1/4,-1/8) node {\textbullet};
  \draw (1/4,-1/8-.1) node {$z$};
  
  \draw (1/4,-15/32) node {\textbullet};
  \draw (1/4,-15/32-.1) node {$w$};

  \draw[dashed] (-1,1/3) -- (-1/2,-1/10) -- (0,-1/7) -- (1/4,-1/8) -- (1/2,-1/12) -- (1,1/5);
      
  \draw[dashed] (-1,-1/4) -- (-1/4,-2/3) -- (1/4,-2/3) -- (1/2,-3/5) -- (1,-1/4);
  
\end{tikzpicture}
  \caption{A diagram of the level sets in the proof of \lref{lowerbarrier}.}
  \label{f.lowerbarrier}
\end{figure}

\begin{proof}
  Step 1.  We first handle the case $0 < \sigma' < 1 - \lambda < 1$, which is symmetric to the previous lemma.  We use \lref{compareup} to show that the event in \eref{stricttouchbelow} is contained in polynomially many events that are controlled by \tref{semiconvex}.  We may assume that $m \geq C$ is large in what follows.

  Define, for $z \in B_{2/3}$, the map
  \begin{equation*}
    \tau_z(x^d, x_d) = \left( m^{\frac{1}{d+1}} (x^d - z^d), m^{\frac{2}{d+1}} (x_d - z_d - z^d \cdot (x^d - z^d)) \right),
  \end{equation*}
  where $x^d = (x_1, ..., x_{d-1})$.  Observe that $\det D \tau_z = m$, $\tau_z(0) = 0$, and $\tau_z(\{ \varphi > \varphi(z) \}) = P$.  Moreover, for any $r \leq c m^{1/(d+1)}$, observe that $(\pi \circ \tau_z)^{-1}(Q_r) \subseteq B_1$.

  Let $\ep > 0$ be a universal constant determined later.  If the event in \eref{stricttouchbelow} occurs, then we can choose $z, w \in B_{2/3}$, $r > 0$, and $n \in \N$ such that
  \begin{equation*}
    r = \ep m^{1/(d+1)},
  \end{equation*}
  \begin{equation*}
    z \in (\pi \circ \tau_w)^{-1}(Q_{2r} \setminus Q_r),
  \end{equation*}
  \begin{equation*}
    \{ \varphi \geq \varphi(w + m^{-2/(d+1)} r e_d) \} \cap B_{2/3} \subseteq \{ h \geq h(z) \},
  \end{equation*}
  \begin{equation*}
    \{ \varphi \geq \varphi(w) \} \cap B_{2/3} \subseteq \{ h \geq h(z) - n \},
  \end{equation*}
  and
  \begin{equation*}
    n \leq \alpha (1 - \tfrac12 \lambda) r.
  \end{equation*}
  This is depicted in \fref{upperbarrier}.  Moreover, we may select $w$ from a predetermined list of $C (m \lambda^{-1} \ep^{-1})^C$ many points in $B_{2/3}$.

  We now reduce to a large deviation event parameterized by $w$.  Applying \lref{comparedown} to the point set $X \cap \{ h \geq h(z) - n \}$ and making $\ep > 0$ sufficiently small, we obtain
  \begin{equation*}
    \inf_{\partial^+ Q_r} s_{(\pi \circ \tau_w)(X) \cap Q_{Cr}} \leq \alpha (1 - \tfrac14 \lambda r).
  \end{equation*}
  Let $\tilde E_w$ denote the above event.  Since, when $m > 0$ is large and $\ep > 0$ is small, $(\pi \circ \tau_w)(X) \cap Q_{Cr} \sim Poisson(\id_{Q_{Cr}})$, \tref{semiconvex} implies
  \begin{equation*}
    \P[\tilde E_w] \leq C \exp( - \lambda^{2/3} c (\log r)^{-2} r^{1/3}).
  \end{equation*}
  Recalling that $r = \ep m^{1/(d+1)}$ and summing over the polynomially many $w$ yields the lemma.

  Step 2.  In the case $0 \leq \sigma' \leq 1 - \lambda < 1$, we first prove
  \begin{multline}
    \label{e.strictertouchbelow}
    \P[\inf_{B_1 \setminus B_{1/3}} (h - \psi) \geq 1 + \inf_{B_{1/3}} (h - \psi) \mbox{ for some } h \in \mathcal{H}(X)] \\ \leq C \exp( - c \lambda^{2/3} (\log m)^{-2} m^{1/3(d+1)}).
  \end{multline}
  This is essentially immediate once we observe that the bound \eref{stricttouchbelow} established in step 1 only needs the assumption $0 < \sigma'$ to hold qualitatively.  We observe that, if the event in \eref{strictertouchbelow} holds, then the event in \eref{stricttouchbelow} holds for $\tilde \sigma(s) = \sigma(s) + (2 \alpha m^{\frac{2}{d+1}})^{-1} s$.  Thus, provided $m \geq C$, we have $0 < \tilde \sigma < 1 - \tfrac12 \lambda$ and can apply step 1.

  Step 3.  We now establish the result for $0 \leq \sigma' \leq 1 - \lambda < 1$.  We make an additional appeal to the Poisson law of $X$.  First, by a standard covering argument, we may replace the outer ball $B_1$ with the ball $B_d$ to give ourselves more room to work.  Second, we may assume that $\{ \sigma = 0 \} = \{ \sigma' = 0 \} = [0,\infty)$.

  We now constrain the geometry of the set where the infimum in \eref{stricttouchbelow} is achieved. Fix $\delta > 0$ and observe that
  \begin{equation}
    \label{e.minbound}
    \P[\min_{B_\delta(x)} h > \min_{B_{\delta/2}(x)} h \mbox{ for all } B_\delta(x) \subseteq B_d] \geq 1 - C \exp( - c \delta^d m).
  \end{equation}
  Indeed, if $\min_{B_\delta(x)} h = \min_{B_{\delta/2}(x)} h$ holds, then $X \cap B_\delta(x)$ has no points on one side of a hyperplane intersecting $B_{\delta/2}(x)$.  For fixed $B_\delta(x)$, the Poisson law gives an upper bound of $C \exp( - c \delta^d m)$ on the probability of this occuring.  The bound \eref{minbound} follows by covering $B_d$ with polynomially many small balls and computing a union bound.

  The event \eref{minbound} excludes the possibility that the infimum in \eref{stricttouchbelow} occurs at some $x$ with $B_\delta(x) \subseteq \{ \psi = 0 \}$.  That is, when the event in \eref{minbound} holds, then $x \in B_{d-\delta}$ and $\inf_{B_d}(h - \psi) = (h - \psi)(x)$ implies $x \in \{ \varphi < C \delta \}$.

  Next, observe that we can choose affine $a : \R^d \to \R^d$ and $\tau \in C^\infty(\R)$ such that $|Da - I| \leq C \delta$, $|\tau' - 1| \leq C \delta$, $\tau'' \leq 0$, and $\tilde \psi = \alpha m^{\frac{2}{d+1}} \tau \circ \sigma \circ \varphi \circ a$ satisfies
  \begin{equation*}
    \tilde \psi \geq \psi \quad \mbox{in } B_{1/3}
  \end{equation*}
  and
  \begin{equation*}
    \tilde \psi < \psi - 1 \quad \mbox{in } \{ \varphi < C \delta \} \setminus B_{d/2}.
  \end{equation*}
  In particular, if the events in \eref{stricttouchbelow} and \eref{minbound} hold, then
  \begin{equation*}
    \inf_{B_d \setminus B_{d/2}} (h - \tilde \psi) \geq 1 + \inf_{B_{1/3}} (h - \tilde \psi).
  \end{equation*}  
  Here we used the integrality of $h$ to conclude the inequality on $\{ \varphi \geq C \delta \} \setminus B_{d/2}$.

  Set $\delta = \ep \lambda$ and rescale \eref{strictertouchbelow} by the affine map $a$.  Since $a$ is within $C \ep \lambda$ of the identity, making $\ep > 0$ small universal allows us to conclude \eref{stricttouchbelow}.
\end{proof}

\subsection{Scaling limit}

We now prove our main theorem by combining the piecewise approximations from \sref{viscosity} with the above barrier lemmas.  The essential idea is that piecewise subsolutions and supersolutions form global barriers for the convex peeling.  The only remaining difficulty is to incorporate the arbitrary weight density.  For this, we use a standard stochastic domination trick.

\begin{lemma}
  If $Y \sim Poisson(\id_{\R^d \times (0,\infty)})$, $f \in L^1_{loc}(\R^d)$, and $f \geq 0$, then
  \begin{equation*}
    Y_f = \{ x \in \R^d : (x,y) \in Y \mbox{ for some } y \in (0,f(x)) \} \sim Poisson(f).
  \end{equation*}
  Moreover, if $f \leq g$, then $Y_f \subseteq Y_g$. \qed
\end{lemma}

The above lemma provides us with a means of locally and monotonically approximating a Poisson cloud with varying density by a Poisson cloud of constant density.  That is, to bound $h_{Y_{mf}}$ from above, it is enough to bound $h_{Y_{mg}}$ from above for some piecewise constant $g \geq f$.  Similarly for bounding below.

\begin{proof}[Proof of \tref{main}]
  Let $U \subseteq \R^d$ be open bounded and convex, $f \in C(\bar U)$ be positive, and let $u \in C(\bar U)$ be the unique solution of \eref{bvpF}.  By \tref{piecewise}, we may select a piecewise subsolution and supersolution $\underline u, \overline u \in C(\bar U)$ of \eref{bvpF} such that $u - \ep \leq \underline u \leq u \leq \overline u \leq u + \ep$.

  Let $X_m = Y_{m f} \sim Poisson(m f)$.  Using the definition of piecewise supersolution, we can cover the event
  \begin{equation*}
    \sup (m^{-2/(d+1)} h_{X_m} - \alpha u) > \ep
  \end{equation*}
  with finitely many simple events as follows.  Let $\psi_k = \sigma_k \circ \varphi \circ a_k \in C^\infty(U_k)$ \red for $1 \leq k \leq N_\varepsilon$ \nc denote the finitely many upper test functions that make up $\overline u$, where $U_k$ is an open ellipsoid such that $U_k = a_k^{-1}(B_{r_k}(x_k))$ for some ball $B_{r_k}(x_k)$.  Thus, if $h_{X_m}$ is too large, then there must be some $k$ such that
  \begin{equation}
    \label{e.kpiecetouch}
    \sup_{a_k^{-1}(B_{r_k}(x_k))} (h_{X_m} - m^{2/(d+1)} \alpha \psi_k) = \sup_{a_k^{-1}(B_{r_k/3}(x_k))} (h_{X_m} - m^{2/(d+1)} \alpha \psi_k).
  \end{equation}
  Recall that
  \begin{equation*}
    F(D \psi_k, D^2 \psi_k) > \sup_{a_k^{-1}(B_{r_k}(x_k)) \cap U} f^2 \quad \mbox{in } B_{r_k}(x_k).
  \end{equation*}
  Let $s_k = \sup_{B_{r_k}(x_k) \cap U} f^2$ and choose $\lambda_k > 0$ such that
  \begin{equation*}
    F(D \psi_k, D^2 \psi_k) = (\sigma_k' \circ \varphi_k \circ a_k) \geq s_k (1 + \lambda_k) \quad \mbox{in } a_k^{-1}(B_{r_k}(x_k)).
  \end{equation*}
  Since $X_m \cap a_k^{-1}(B_{r_k}(x_k)) \subseteq X_{m,k} := Y_{m s_k \id_{a_k^{-1}(B_{r_k}(x_k))}}$, we see that the event \eref{kpiecetouch} is contained in the event that $X_{m,k}$ has a local height function $h$ on $a_k^{-1}(B_{r_k}(x))$ such that
  \begin{equation*}
    \sup_{a_k^{-1}(B_{r_k}(x_k))} (h - m^{2/(d+1)} \alpha \psi_k) = \sup_{a_k^{-1}(B_{r_k/3}(x_k))} (h - m^{2/(d+1)} \alpha \psi_k).
  \end{equation*}
  This is equivalent to
  \begin{equation*}
    \sup_{B_{r_k}(x_k)} (h \circ a_k^{-1} - m^{2/(d+1)} \alpha \sigma_k \circ \varphi) = \sup_{B_{r_k/3}(x_k)} (h \circ a_k^{-1} - m^{2/(d+1)} \alpha \sigma_k \circ \varphi).
  \end{equation*}
  Applying \lref{upperbarrier}, this event has probability bounded by 
  \[C_k \exp(-c_k (\log m)^{-2} m^{1/3(d+1)}).\]
 \red Summing over $k$ the probability is bounded by
\[\sum_{k=1}^{N_\eps} C_k \exp(-c_k (\log m)^{-2}m^{1/3(d+1)}) \leq N_\eps Q_\eps\exp(-q_\eps(\log m)^{-2}m^{1/3(d+1)}),\]
where  $q_\eps:=\min_{1\leq k\leq N_\eps}c_k$ and $Q_\eps:=\max_{1\leq k\leq N_\eps}C_k$. \nc The subsolution bound is identical, using \lref{lowerbarrier} in place of \lref{upperbarrier}.
\end{proof}

\begin{proof}[Proof of \cref{main}]
Let $X_m \sim Poisson(m f)$. Conditioned on $\# X_m = k$, $h_{X_m}$ and $h_{Z_k}$ have the same distribution. Since $\# X_m$ is Poisson with mean $m$ we have
\[\P[\sup_{\bar U} |m^{-2/(d+1)} h_{Z_m} - \alpha h| > \ep] \leq \frac{m!e^{m}}{m^m}\P[\sup_{\bar U} |m^{-2/(d+1)} h_{X_m} - \alpha h|> \ep].\]
The proof is completed with an application of a version of Stirling's formula $m! e^m \leq em^{m+1/2}$ for $m\geq 1$.
\end{proof}

\end{document}